\documentclass{psp}

\usepackage{amsmath,amssymb,theoremref,float,mathtools}
\usepackage[shortlabels]{enumitem}

\usepackage{xypic} %for diagrams
\usepackage{mathrsfs}

\newromanexpr\Hess{Hess}
\newtheorem{thm}{Theorem}[section] % 1st argument is your name for it
\newtheorem{lem}[thm]{Lemma}     % 2nd argument is what is printed
\newtheorem{cor}[thm]{Corollary}
\newtheorem{prop}[thm]{Proposition}
\newtheorem{claim}[thm]{Claim}
\newtheorem{fact}[thm]{Fact}

\newnumbered{assertion}[thm]{Assertion}    % 1st argument is your name for it
\newnumbered{conjecture}[thm]{Conjecture}  % 2nd argument is what is printed
\newnumbered{defn}[thm]{Definition}
\newnumbered{hypothesis}[thm]{Hypothesis}
\newnumbered{rmk}[thm]{Remark}
\newnumbered{note}[thm]{Note}
\newnumbered{obs}[thm]{Observation}
\newnumbered{problem}[thm]{Problem}
\newnumbered{quest}[thm]{Question}
\newnumbered{algorithm}[thm]{Algorithm}
\newnumbered{example}[thm]{Example}
\newunnumbered{notation}[thm]{Notation}

\newcommand{\Zb}{\mathbb{Z}}
\newcommand{\Nb}{\mathbb{N}}

\newcommand{\Uca}{\mathcal{U}}

\newcommand{\Qp}{\mathbb{Q}_p}

\newcommand{\mc}[1]{\mathcal{#1}}
\newcommand{\mb}[1]{\mathbb{#1}}
\newcommand{\mf}[1]{\mathfrak{#1}}
\newcommand{\ms}[1]{\mathscr{#1}} %requires mathrsfs

\newcommand{\sleq}{\leqslant}
\newcommand{\sgeq}{\geqslant}

\newcommand{\Es}{\mathscr{E}}

\newcommand{\Aut}{\mathop{\rm Aut}}
\newcommand{\Inn}{\mathop{\rm Inn}}

\newcommand{\cgrp}[1]{\overline{\left\langle #1 \right\rangle}}
\newcommand{\grp}[1]{\left\langle #1 \right\rangle}
\newcommand{\ol}[1]{\overline{#1}}

\newcommand{\Rad}[1]{\mathop{\rm Rad}_{#1}\nolimits}
\newcommand{\Res}[1]{\mathop{\rm Res}_{#1}\nolimits}

\newcommand{\Qi}[1]{\mathbb{Q}_{#1}}

\begin{document}

\title[T.d.l.c. groups locally of finite rank]{Totally disconnected locally compact groups locally of finite rank}
\author[Phillip Wesolek]{PHILLIP WESOLEK\\
  Universit\'{e} catholique de Louvain,
  Louvain-la-Neuve, Belgium \addressbreak
  e-mail\textup{: \texttt{phillip.wesolek@uclouvain.be}}}

\receivedline{Received \textup{24} June \textup{2014}; revised \textup{27} January \textup{2015}}

\maketitle

\begin{abstract}
We study totally disconnected locally compact second countable (t.d.l.c.s.c.) groups that contain a compact open subgroup with finite rank. We show such groups that additionally admit a pro-$\pi$ compact open subgroup for some finite set of primes $\pi$ are virtually an extension of a finite direct product of topologically simple groups by an elementary group. This result, in particular, applies to l.c.s.c. $p$-adic Lie groups. We go on to obtain a decomposition result for all t.d.l.c.s.c. groups containing a compact open subgroup with finite rank. In the course of proving these theorems, we demonstrate independently interesting structure results for t.d.l.c.s.c. groups with a compact open pro-nilpotent subgroup and for topologically simple l.c.s.c. $p$-adic Lie groups. 
\end{abstract}

\section{Introduction}

\indent There are a number of theorems that point to a deep relationship between the structure of the compact open subgroups and the global structure of a totally disconnected locally compact (t.d.l.c.) group; cf. \cite{BEW11}, \cite{BM00}, \cite{CRW_1_13}, \cite{CRW_2_13}, \cite{Will07}. These global structural consequences of local properties are, following M. Burger and S. Mozes \cite{BM00}, often called \emph{local-to-global} structure theorems. In the work at hand, we contribute to the body of local-to-global structure theorems by proving results for t.d.l.c. groups that have a compact open subgroup of finite rank.
\begin{defn}
A profinite group has \textbf{rank} $0<r\sleq\infty$ if every closed subgroup contains a dense $r$-generated subgroup. When a profinite group has rank $r<\infty$, we say it has \textbf{finite rank}.
\end{defn}
\noindent These t.d.l.c. groups are of wide interest as locally compact $p$-adic Lie groups have a compact open subgroup with finite rank.\par

\begin{rmk} We study t.d.l.c. groups that are also second countable (s.c.). The second countability assumption is natural and mild. T.d.l.c.s.c. groups belong to the robust and natural class of Polish groups studied in descriptive set theory. T.d.l.c.s.c. groups are also the correct generalization of countable discrete groups studied by geometric group theorists.  More pragmatically, most natural examples of t.d.l.c. groups are second countable. As for the mildness of our assumption, t.d.l.c. groups may always be written as a directed union of compactly generated open subgroups, and these subgroups are second countable modulo a compact normal subgroup \cite[(8.7)]{HR79}. The study of t.d.l.c. groups therefore reduces to the study of second countable groups and profinite groups, so little generality is lost. 
\end{rmk}

\subsection{Statement of results}

\begin{defn} 
A t.d.l.c. group $G$ is said to be \textbf{locally} $x$ for $x$ a property of profinite groups, if $G$ contains a compact open subgroup with property $x$. In the case $G$ has a compact open subgroup with finite rank, we say $G$ is \textbf{locally of finite rank}.
\end{defn}

We consider t.d.l.c.s.c. groups that are locally of finite rank. Via a deep result of M. Lazard \cite{L65} and later work of A. Lubotzky and A. Mann \cite{LM87}, l.c.s.c. $p$-adic Lie groups, e.g. $SL_3(\Qp)$, are examples. Our first theorem is a structure result for locally of finite rank groups that are also locally pro-$\pi$ for some finite set of primes $\pi$. The statement requires the notion of an elementary group: The collection of elementary groups is the smallest class of t.d.l.c.s.c. groups that contains the second countable profinite groups and the countable discrete groups, is closed under group extension, and is closed under countable increasing union. This class is identified and investigated in \cite{W_1_14}.
\begin{thm}
Suppose $G$ is a t.d.l.c.s.c. group that is locally of finite rank and locally pro-$\pi$ for some finite set of primes $\pi$. Then either $G$ is elementary or there is a series of closed characteristic subgroups
\[
 \{1\}\sleq A_1\sleq A_2\sleq G
\]
such that 
\begin{enumerate}[(1)]
\item~$A_1$ is elementary and $G/A_2$ is finite; and

\item~$A_2/A_1\simeq N_1\times\dots\times N_k$ for $0< k<\infty$ where each $N_i$ is a non-elementary compactly generated topologically simple $p$-adic Lie group of adjoint simple type for some $p\in \pi$.
\end{enumerate}
\end{thm}
A $p$-adic Lie group is of adjoint simple type if it is isomorphic to $S(\Qp)^+$ for some adjoint $\Qp$-simple isotropic $\Qp$-algebraic group $S$.

\indent As a corollary, we obtain a decomposition result for l.c.s.c. $p$-adic Lie groups. We call attention to the similarity with the solvable-by-semisimple decomposition for connected Lie groups.

\begin{cor}
Suppose $G$ is l.c.s.c. $p$-adic Lie group. Then either $G$ is elementary or there is a series of closed characteristic subgroups
\[
\{1\}\sleq A_1\sleq A_2\sleq G
\]
such that 
\begin{enumerate}[(1)]
\item~$A_1$ is elementary and $G/A_2$ is finite; and
 
\item~$A_2/A_1\simeq N_1\times\dots\times N_k$ for $0< k<\infty$ where each $N_i$ is a non-elementary compactly generated topologically simple $p$-adic Lie group of adjoint simple type.
\end{enumerate}
\end{cor}
For a non-elementary l.c.s.c. $p$-adic Lie group $G$, it follows $G/A_1$ has a semisimple Lie algebra. The aforementioned similarity with the connected setting is thus quite deep.\par

\indent We then relax the locally pro-$\pi$ assumption. We first note an immediate corollary of our results for locally pro-$\pi$ groups.
\begin{cor}
Suppose $G$ is a t.d.l.c.s.c. group that is locally of finite rank. Then either $G$ is elementary or there is an $\subseteq$-increasing exhaustion $(P_i)_{i\in \omega}$ of $G$ by compactly generated open subgroups such that for each $i$ there is a series of closed characteristic subgroups
\[
\{1\}\sleq A_1(i)\sleq A_2(i)\sleq P_i
\] 
such that
\begin{enumerate}[(1)]
\item~$A_1(i)$ is elementary and $P_i/A_2(i)$ is finite; and

\item~$A_2(i)/A_1(i)\simeq N_1(i)\times\dots\times N_{k_i}(i)$ for $0< k_i<\infty$ where each $N_j(i)$ is a non-elementary compactly generated topologically simple $p$-adic Lie group of adjoint simple type for some prime $p$.
\end{enumerate}
\end{cor}

\indent We go on to obtain a more detailed structure result.
\begin{thm}
Suppose $G$ is a t.d.l.c.s.c. group that is locally of finite rank. Then either $G$ is elementary or there is a series of closed characteristic subgroups
\[
\{1\}\sleq A_1\sleq A_2\sleq G
\]
such that 
\begin{enumerate}[(1)]
\item~$A_1$ and $G/A_2$ are elementary; and

\item~there is a possibly infinite set of primes $\pi$ so that $A_2/A_1$ is a quasi local direct product of $(N_p,W_p)_{p\in \pi}$ where $N_p$ is a finite direct product of non-elementary compactly generated topologically simple $p$-adic Lie groups of adjoint simple type and $W_p$ is a compact open subgroup of $N_p$.
\end{enumerate}
\end{thm}
A group $G$ is a quasi local direct product of $(N_p,W_p)_{p\in \pi}$ if each $N_p$ is a closed normal subgroup of $G$ and the multiplication map $m:\bigoplus_{p\in\pi }(N_p,W_p)\rightarrow G$ is a well-defined injective homomorphism with dense image. The group $\bigoplus_{p\in\pi }(N_p,W_p)$ is the local direct product of $(N_p)_{p\in \pi}$ over $(W_p)_{p\in \pi}$. \par

\indent The proofs of the above theorems require two independently interesting lines of inquiry. We first study locally pro-nilpotent groups. Here we obtain a compelling decomposition result.

\begin{thm}
Suppose $G$ is a t.d.l.c.s.c. group that is locally pro-nilpotent. Then either $G$ is elementary or there is a series of closed characteristic subgroups
\[
\{1\}\sleq A_1\sleq A_2\sleq G
\]  
so that
\begin{enumerate}[(1)]
\item~$A_1$ and $G/A_2$ are elementary; and

\item~there is a possibly infinite set of primes $\pi$ so that $A_2/A_1\simeq \bigoplus_{p\in\pi }(L_p,U_p)$ where $L_p$ is a locally pro-$p$ t.d.l.c.s.c. group and $U_p$ is a compact open subgroup of $L_p$.
\end{enumerate}
\end{thm}

\indent Our principal local-to-global results additionally require a study of topologically simple l.c.s.c. $p$-adic Lie groups. We prove three results for such groups. 
\begin{prop}
If $G$ is a non-elementary topologically simple l.c.s.c. $p$-adic Lie group, then $G$ is isomorphic to a closed subgroup of $GL_n(\Qp)$ for some $n$ and $L(G)$, the Lie algebra of $G$, is simple.
\end{prop}

Via the proposition, a close relationship with $\Qp$-algebraic groups emerges.
\begin{thm}
Suppose $G$ is a non-elementary topologically simple l.c.s.c. $p$-adic Lie group. Then $G\simeq H(\Qp)^+$ with $H$ an adjoint $\Qp$-simple isotropic $\Qp$-algebraic group. 
\end{thm}

This relationship gives strong restrictions on the automorphism group.
\begin{thm}
Suppose $G$ is a non-elementary topologically simple l.c.s.c. p-adic Lie group, let $\Aut(G)$ be the group of topological group automorphisms of $G$, and let $\Inn(G)$ be the collection of inner automorphisms. Then $\Aut(G)/\Inn(G)$ is finite.
\end{thm}

\begin{acknowledgements} 
Many of the results herein form part of author's thesis work at the University of Illinois at Chicago. The author thanks his thesis adviser Christian Rosendal and the University of Illinois at Chicago. The author also thanks Pierre-Emmanuel Caprace and Ramin Takloo-Bighash for their many thoughtful comments and helpful suggestions. The author finally thanks the anonymous referee for suggesting a generalization that gave Proposition~\ref{prop:loc_nil} and for pointing out results in the literature that streamlined many of the proofs.
\end{acknowledgements}

\section{Generalities on t.d.l.c. groups}
We begin with a brief overview of necessary background. The notations, definitions, and facts discussed here are used frequently and, typically, without reference.

\subsection{Notations} All groups are taken to be Hausdorff topological groups and are written multiplicatively. Topological group isomorphism is denoted $\simeq$. We use ``t.d.", ``l.c.", and ``s.c." for ``totally disconnected", ``locally compact", and ``second countable", respectively.\par

\indent For a topological group $G$, $S(G)$ and $\Uca(G)$ denote the collection of closed subgroups of $G$ and the collection of compact open subgroups of $G$, respectively. All subgroups are taken to be closed unless otherwise stated. We write $H\leqslant _oG$ and $H\sleq_{cc} G$ to indicate $H$ is an open subgroup of $G$ and $H$ is a cocompact subgroup of $G$, respectively. Recall $H\sleq G$  is cocompact if the quotient space $G/H$ is compact in the quotient topology. 

\indent For any subset $K\subseteq G$, $C_G(K)$ is the collection of elements of $G$ that centralize every element of $K$. We denote the collection of elements of $G$ that normalize $K$ by $N_G(K)$. The topological closure of $K$ in $G$ is denoted by $\ol{K}$. For $A,B\subseteq G$, we put 
\[
\begin{array}{c}
A^B:=\left\{bab^{-1}\mid a\in A\text{ and }b\in B\right\} \text{ and } \\
\left[A,B\right]:=\grp{aba^{-1}b^{-1}\mid a\in A\text{ and }b\in B}.
\end{array}
\]
For $k\sgeq 1$, $A^{\times k}$ denotes the $k$-th Cartesian power. For $a,b\in G$, $[a,b]:=aba^{-1}b^{-1}$.\par

\indent Ordinal numbers are used in this work. The first countable transfinite ordinal is denoted by $\omega$; as we assume the natural numbers $\Nb$ contain zero, $\omega=\Nb$ as linear orders. We use $\Nb$ and $\omega$ interchangeably.\par

\indent We use $\mathbb{P}$ to denote the prime numbers. For $\pi\subseteq \mathbb{P}$, we put $\pi':=\mathbb{P}\setminus \pi$. When $\pi=\{p\}$, the set $\pi'$ is written $p'$.

\subsection{Basic theory} 
The foundation of the theory of t.d.l.c. groups is an old result of D. van Dantzig:

\begin{thm}[van Dantzig {\cite[(7.7)]{HR79}}]
 A t.d.l.c. group admits a basis at $1$ of compact open subgroups.
\end{thm}
It follows the compact open subgroups in a t.d.l.c. groups given by van Dantzig's theorem are profinite, i.e. inverse limits of finite groups. Profinite groups will be discussed at length later. For the moment, we remark that compact groups and profinite groups are one and the same in the category of t.d.l.c groups. We thus use ``compact group" and ``profinite group" interchangeably. We remark further that a profinite group is pro-$\pi$ for $\pi$ some set of primes $\pi$ if every finite continuous quotient is a $\pi$-group - that is the order is divisible by only primes in $\pi$.

\indent Pro-$\pi$ groups with $\pi$ a finite set of primes play an important role in the study of compactly generated groups.

\begin{thm}[{Caprace, see \cite[Proposition 4.9]{CRW_2_13}}]\label{thm:Caprace_loc_pro_pi}
If $G$ is a compactly generated t.d.l.c. group, then for every compact open subgroup $U$ there is a compact normal $K\trianglelefteq G$ so that $K\sleq U$ and that $G/K$ is locally pro-$\pi$ for some finite set of primes $\pi$.
\end{thm}

\indent Topological analogues of the familiar isomorphism theorems hold for t.d.l.c.s.c. groups. The first isomorphism theorem requires non-trivial modification, hence we recall its statement.

\begin{thm}[{\cite[(5.33)]{HR79}}]
Let $G$ be a t.d.l.c.s.c. group, $A\sleq G$ be a closed subgroup, and $H\trianglelefteq G$ be a closed normal subgroup. If $AH$ is closed, then $AH/H\simeq A/(A\cap H)$ as topological groups.
\end{thm}

\indent In the category of t.d.l.c.s.c. groups, care must be taken with infinite unions. Suppose $(G_i)_{i\in \omega}$ is an increasing sequence of t.d.l.c.s.c. groups such that $G_i\sleq_oG_{i+1}$ for each $i$. The group $G:=\bigcup_{i\in \omega} G_i$ is then a t.d.l.c.s.c. group under the \textbf{inductive limit topology}: $A\subseteq G$ is defined to be open if and only if $A\cap G_i$ is open in $G_i$ for each $i$. \par

\indent Using our notion of an infinite union, we may define an infinite direct product that stays in the category of t.d.l.c.s.c. groups. This definition goes back to J. Braconnier.

\begin{defn} 
Suppose $A$ is a countable set, $(G_a)_{a\in A}$ is a sequence of t.d.l.c.s.c. groups, and for each $a\in A$ there is a distinguished $U_a\in \Uca(G_a)$. Letting $\{a_i\}_{i\in \omega}$ enumerate $A$, put
\begin{enumerate}[$\bullet$]
\item~ $S_0:=\prod_{i\in \omega}U_{a_i}$ and give $S_0$ the product topology.

\item~ $S_{n+1}:=G_{a_0}\times\dots\times G_{a_n}\times \prod_{i\sgeq n+1}U_{a_i}$ and give $S_{n+1}$ the product topology.
\end{enumerate}
The \textbf{local direct product of $(G_a)_{a\in A}$ over $(U_a)_{a\in A}$} is defined to be 
\[
\bigoplus_{a\in A}\left(G_a,U_a\right):=\bigcup_{i\in \omega}S_i
\]
with the inductive limit topology. 
\end{defn}

Since $S_n\sleq_o S_{n+1}$ for each $n\in \omega$, the group $\bigoplus_{a\in A}(G_a,U_a)$ is a t.d.l.c.s.c. group with $\prod_{a\in A}U_a$ as a compact open subgroup. The isomorphism type of a local direct product is also independent of the enumeration of $A$ used in the definition.\par

\indent There is a weakening of the notion of a direct product: A t.d.l.c.s.c. group $G$ is a \textbf{quasi-product} with quasi-factors $N_1,\dots,N_k$ if each $N_i$ is a closed normal subgroup of $G$ and the multiplication map $N_1\times\dots \times N_k\rightarrow G$ is injective with dense image. This notion naturally generalizes to local direct products: A group $G$ is a \textbf{quasi local direct product} of $(N_i,U_i)_{i\in \omega}$ if $N_i$ is a closed normal subgroup of $G$, $U_i\in \Uca(N_i)$, and the multiplication map $m:\bigoplus_{i\in \omega }(N_i,U_i)\rightarrow G$ is a well-defined injective homomorphism with dense image.

\indent There are two important concepts concerning subgroups of a t.d.l.c.s.c. group $G$. First, following P-E. Caprace, C. Reid, and G. Willis \cite{CRW_1_13}, a subgroup $K\sleq G$ is called \textbf{locally normal} if $K$ is compact and $N_G(K)$ is open. Second, two subgroups $M\sleq G$ and $N\sleq G$ are \textbf{commensurate}, denoted $M\sim_c N$, if $|M:M\cap N|$ and $|N:M\cap N|$ are finite. \par

\indent We require a few additional facts around commensurated subgroups. It is easy to check $\sim_c$ is an equivalence relation on $S(G)$ and is preserved under the action by conjugation of $G$ on $S(G)$. The commensurability relation gives rise to an additional subgroup: For $N\leqslant G$, the \textbf{commensurator subgroup} of $N$ in $G$ is 
\[
Comm_G(N):=\left\{g\in G\;|\;gNg^{-1}\sim_cN\right\}.
\]
If $Comm_G(N)=G$, we say $N$ is \textbf{commensurated}.\par

\indent We shall make frequent use of two canonical normal subgroups of a t.d.l.c.s.c. group $G$. Generalizing the notion of the centre of a group, the \textbf{quasi-centre}, defined in \cite{BM00}, of $G$ is 
\[
QZ(G):=\left\{g\in G\mid C_G(g)\text{ is open}\right\}.
\]
The group $QZ(G)$ is a characteristic but not necessarily closed subgroup. For the second canonical normal subgroup, a closed subgroup of $G$ is \textbf{locally elliptic} if every finite subset generates a relatively compact subgroup. V.P. Platonov \cite{Plat66} shows there is a unique maximal closed normal subgroup of $G$ that is locally elliptic; this subgroup is called the \textbf{the locally elliptic radical} and is denoted by $\Rad{\mc{LE}}(G)$. The same work demonstrates that a t.d.l.c.s.c. group is locally elliptic if and only if it is a countable increasing union of compact open subgroups. \par

\indent The locally elliptic radical along with Theorem~\ref{thm:Caprace_loc_pro_pi} give a somewhat canonical decomposition for non-compactly generated groups. 
 
\begin{cor}\label{cor:loc_nil_main2}
Suppose $G$ is a t.d.l.c.s.c. group. Then there is an increasing exhaustion $(H_i)_{i\in \omega}$ of $G$ by compactly generated open subgroups so that for each $i$, $H_i/\Rad{\mc{LE}}(H_i)$ is locally pro-$\pi$ for some finite set of primes $\pi$. 
\end{cor}
\noindent We stress that $\pi$ in the above corollary depends on $i$ and, in general, grows as $i$ increases.\par

\indent We conclude by recalling a general technique for producing normal but \emph{not necessarily closed} subgroups of a t.d.l.c.s.c. group $G$. A subset $\mathcal{A}\subseteq S(G)$ is \textbf{conjugation invariant} if $\mathcal{A}$ is fixed setwise under the action by conjugation of $G$ on $S(G)$. If $\mc{A}$ is fixed setwise by all topological group automorphisms of $G$, we say $\mc{A}$ is \textbf{invariant}. We say $\mc{A}$ is \textbf{hereditary} if for all $A\in \mathcal{A}$, $S(A)\subseteq \mathcal{A}$. The $\mathcal{A}$\textbf{-core}, denoted $N_{\mathcal{A}}$, is the collection of $g\in G$ such that for all $C\in \mc{A}$, $\cgrp{g,C}\in\mathcal{A}$.\par

\indent By results of \cite{W_1_14} or as an easy verification, if $\mathcal{A}\subseteq S(G)$ is (invariant) conjugation invariant and hereditary, then $N_{\mathcal{A}}$ is a (characteristic) normal subgroup of $G$. A subgroup of $G$ of the form $N_{\mathcal{A}}$ for some conjugation invariant and hereditary $\mathcal{A}\subseteq S(G)$ is called a \textbf{synthetic subgroup} of $G$. It is easy to see all closed normal subgroups are synthetic subgroups. However, the collection of synthetic subgroups of $G$ often strictly contains the collection of closed normal subgroups of $G$. For example, $QZ(G)$ is a synthetic subgroup and is rarely closed.

\subsection{Profinite groups} 
We now recall a few facts and definitions from the theory of profinite groups; our discussion falls well short of comprehensive. We direct the interested reader to the excellent texts \cite{RZ00} and \cite{Wil98}.\par

\indent  Profinite groups admit a basis at $1$ of open normal subgroups. For a profinite group $U$, we say $(U_i)_{i\in \omega}$ is a \textbf{normal basis at} $1$ for $U$, if $U_0=U$, $(U_i)_{i\in \omega}$ is $\subseteq$-decreasing with trivial intersection, and for each $i$, $U_i\trianglelefteq_o U$. \par

\indent The group $U$ is said to be \textbf{pro-$x$} for $x$ some property of finite groups if $U$ is an inverse limit of finite groups with property $x$. For example, $U$ may be pro-$p$ for some prime $p$ or pro-nilpotent. \par

\indent We say $U$ is \textbf{topologically (finitely generated) $r$-generated} if $U$ admits a dense (finitely generated) $r$-generated subgroup. Central to this work,

\begin{defn}
A profinite group $U$ has \textbf{rank} $0<r\sleq \infty$ if for every closed $H\sleq U$, $H$ is topologically $r$-generated. If a profinite group $U$ has rank $r$ for some $r<\infty$, we say $U$ has \textbf{finite rank}.
\end{defn}

\indent Profinite groups with finite rank have a well understood structure.

\begin{thm}[{\cite[Theorem 8.4.1]{Wil98}}]\label{thm:finrank}
If $U$ is a profinite group with finite rank, then $U$ has a series of normal subgroups $\{1\}\sleq C\sleq N\sleq U$ such that $C$ is pro-nilpotent, $N/C$ is solvable, and $U/N$ is finite.
\end{thm}

\indent Profinite groups have a Sylow theory arising from the inverse limit construction. For a prime $p$, a \textbf{$p$-Sylow subgroup} of a profinite group $U$, denoted $U_p$, is a maximal pro-$p$ subgroup of $U$. Analogous to the finite setting,

\begin{prop}[{\cite[2.2.2]{Wil98}}] Let $U$ be a profinite group and $p$ a prime. Then
\begin{enumerate}[(1)]
\item~$U$ has $p$-Sylow subgroups.

\item~All $p$-Sylow subgroups are conjugate.

\item~Every pro-$p$ subgroup is contained in a $p$-Sylow subgroup.
\end{enumerate}
\end{prop}

\begin{prop}[{\cite[Proposition 2.3.8]{RZ00}}]\label{prop:pro-nil} Suppose $U$ is a profinite group that is pro-nilpotent. Then $U\simeq \prod_{p\in \mb{P}} U_p$. In particular,
\begin{enumerate}[(1)]
\item~For each prime $p$, $U$ has a unique normal $p$-Sylow subgroup $U_p$.

\item~For primes $p\neq q$, $U_p\sleq C_U(U_q)$.
\end{enumerate}
\end{prop}

\indent The \textbf{Frattini subgroup} of $U$, denoted $\Phi(U)$, is the intersection of all maximal proper open subgroups of $U$. 

\begin{prop}[{\cite[Proposition 2.5.1]{Wil98}}]\label{prop:fratt}
Let $U$ be a profinite group. If $H\sleq U$ and $H\Phi(U)=U$, then $H=U$.
\end{prop}

\begin{prop}[{\cite[Proposition 2.8.11]{RZ00}}]\label{prop:s_solvable}
If $U$ is pro-supersolvable and pro-$\pi$ for some finite set of primes $\pi$, then $U$ is topologically finitely generated if and only if $\Phi(U)$ is open. 
\end{prop}

\indent The \emph{$\pi$-core} of a profinite group $U$, $O_{\pi}(U)$, is the closed subgroup generated by all subnormal pro-$\pi$ subgroups of $U$ where $\pi$ is a possibly infinite set of primes. In \cite[Lemma 2.4]{R13}, $O_{\pi}(U)$ is shown to be pro-$\pi$ and normal. Under certain assumptions on $U$, the $p'$-core behaves nicely where $p'$ denotes the collection of primes different from $p$.

\begin{thm}[Reid {\cite[Corollary 5.11]{R13}}]\label{thm:reidcore}
If $U$ is a profinite group such that $U$ has a topologically finitely generated $p$-Sylow subgroup, then $U/O_{p'}(U)$ is virtually pro-$p$.
\end{thm}

\indent Finitely generated $p$-Sylow subgroups have strong structural consequences by work of O.V. Melnikov; we include a proof via Theorem~\ref{thm:reidcore} for completeness.

\begin{thm}[{Melnikov, \cite{M96}}]\label{thm:melnikov} If $U$ is a profinite group that is pro-$\pi$ for some finite set of primes $\pi$ and for which every $p$-Sylow subgroup is topologically finitely generated, then $U$ is virtually pro-nilpotent.
\end{thm}
\begin{proof}
Let $p_1,\dots,p_n$ list $\pi$ and form $O_{p'_1}(U),\dots,O_{p'_n}(U)$. Since $U$ is pro-$\pi$, 
\[
\bigcap_{i=1}^nO_{p'_i}(U)=\{1\},
\]
hence the diagonal map $d:U\rightarrow U/O_{p'_1}(U)\times \dots \times U/O_{p'_n}(U)$ is injective. In view of Theorem~\ref{thm:reidcore}, there is $L_i\sleq_o U/O_{p'_i}(U)$ so that $L_i$ is pro-$p_i$. The group 
\[
V:=d^{-1}(L_1\times \dots \times L_n)
\]
is thus an open subgroup of $U$, and $V\simeq d(V)\sleq L_1\times \dots \times L_n$. Since $L_1\times \dots \times L_n$ is pro-nilpotent, we conclude $V$ is also pro-nilpotent verifying the theorem.
\end{proof}

Melnikov's theorem implies a useful corollary.

\begin{cor}\label{cor:melnikov}
If $U$ is a finite rank profinite group that is pro-$\pi$ for some finite set of primes $\pi$, then $U$ is virtually pro-nilpotent.
\end{cor}

\indent Associated to a profinite group $U$ is the group of continuous automorphisms of $U$, denoted $\Aut(U)$. There is a natural topology on $\Aut(U)$: For $K\trianglelefteq U$, put 
\[
A_U(K):=\left\{\phi\in \Aut(U)\mid\phi(g)g^{-1}\in K\text{ for all }g\in U\right\}.
\]
By declaring the sets $A_U(K)$ as $K$ varies over open normal subgroups of $U$ to be a basis at $1$, $\Aut(U)$ becomes a topological group. Following L. Ribes and P. Zalesskii \cite{RZ00}, we call this topology the \textbf{congruence subgroup topology} of $\Aut(U)$. In the case $U$ is topologically finitely generated, $\Aut(U)$ is a profinite group under the congruence subgroup topology \cite[Corollary 4.4.4]{RZ00}.

\subsection{Elementary groups}
Elementary groups play a central role in this work. The class of elementary groups is, intuitively, the class of all t.d.l.c.s.c. groups that can reasonably be built by hand from second countable profinite groups and countable discrete groups. Formally, 

\begin{defn}
The class of \textbf{elementary groups} is the smallest class $\Es$ of t.d.l.c.s.c. groups such that
\begin{enumerate}[(i)]

\item~$\Es$ contains all second countable profinite groups and countable discrete groups.

\item~$\Es$ is closed under taking group extensions of second countable profinite or countable discrete groups. I.e. if $G$ is a t.d.l.c.s.c. group and $H\trianglelefteq G$ is a closed normal subgroup with $H\in \Es$ and $G/H$ profinite or discrete, then $G\in \Es$.

\item~If $G$ is a t.d.l.c.s.c. group and $G=\bigcup_{i\in \omega}O_i$ where $(O_i)_{i\in \omega}$ is an $\subseteq$-increasing sequence of open subgroups of $G$ with $O_i\in\Es$ for each $i$, then $G\in\Es$. We say $\Es$ is \textbf{closed under countable increasing unions}.
\end{enumerate}
\end{defn}

The class of elementary groups is surprisingly robust supporting our intuition that $\Es$ is the class of groups ``built by hand".
\begin{thm}[{\cite[Theorem 3.18]{W_1_14}}]\label{thm:closure_main} $\Es$ enjoys the following permanence properties:
\begin{enumerate}[(1)]
\item~$\Es$ is closed under group extension. 

\item~If $G\in \Es$, $H$ is a t.d.l.c.s.c. group, and $\psi:H\rightarrow G$ is a continuous, injective homomorphism, then $H\in \Es$. In particular, $\Es$ is closed under taking closed subgroups.

\item~$\Es$ is closed under taking quotients by closed normal subgroups.

\item~If $G$ is a residually elementary t.d.l.c.s.c. group, then $G\in \Es$. In particular, $\Es$ is closed under inverse limits.

\item~$\Es$ is closed under quasi-products.

\item~$\Es$ is closed under local direct products.

\item~If $G$ is a t.d.l.c.s.c. group and there is $(C_i)_{i\in \omega}$ an $\subseteq$-increasing sequence of elementary subgroups of $G$ such that $N_G(C_i)$ is open for each $i$ and $\ol{\bigcup_{i\in \omega}C_i}=G$, then $G\in \Es$. 
\end{enumerate}
\end{thm}

We make use of a strong sufficient condition to be elementary.
\begin{thm}[{\cite[Theorem 8.1]{W_1_14}}]\label{thm:locsolv_e}
If $G$ is a t.d.l.c.s.c. group and $G$ has an open solvable subgroup, then $G\in \Es$.
\end{thm}

The permanence properties of the class of elementary groups give rise to two canonical normal subgroups in an arbitrary t.d.l.c.s.c. group. 
\begin{thm}[{\cite[Theorem 7.9]{W_1_14}}] Let $G$ be a t.d.l.c.s.c. group. Then
\begin{enumerate}[(1)]
\item There is a unique maximal closed normal subgroup $\Rad{\Es}(G)$ such that $\Rad{\Es}(G)$ is elementary. 
\item There is a unique minimal closed normal subgroup $\Res{\Es}(G)$ such that $G/\Res{\Es}(G)$ is elementary.
\end{enumerate}
\end{thm}
We call $\Rad{\Es}(G)$ and $\Res{\Es}(G)$ the \textbf{elementary radical} and \textbf{elementary residual}, respectively. It is easy to verify $G/\Rad{\Es}(G)$ has trivial quasi-centre and has trivial locally elliptic radical. \par

\indent It can be the case $\Rad{\Es}(G)=\{1\}$ and $\Res{\Es}(G)=G$. We give a name to such groups: A t.d.l.c.s.c. group is \textbf{elementary-free} if it has no non-trivial elementary normal subgroups and no non-trivial elementary quotients. Elementary-free groups have a nice property.

\begin{thm}[{\cite[Corollary 9.12]{W_1_14}}]\label{thm:[A]semisimple}
For a t.d.l.c.s.c. group $G$, $G/\Rad{\Es}(G)$ has no non-trivial locally normal abelian subgroups. In particular, elementary-free t.d.l.c.s.c. groups have no non-trivial locally normal abelian subgroups.
\end{thm}

\indent The elementary radical and residual may be used to produce a characteristic series. For a t.d.l.c.s.c. group $G$, the \textbf{ascending elementary series} is defined by $A_0:=\{1\}$, $A_1:=\Rad{\Es}(G)$, $A_2:=\pi^{-1}(\Res{\Es}(G/A_1))$ where $\pi:G\rightarrow G/A_1$ is the usual projection, and $A_3:=G$. The ascending elementary series gives a method of reducing to elementary-free groups.

\begin{thm}[{\cite[Theorem 7.17]{W_1_14}}]\label{thm:acs_E_series}
Let $G$ be a t.d.l.c.s.c. group. Then the ascending elementary series 
\[
\{1\}\sleq A_1\sleq A_2\sleq G
\]
is a series of characteristic subgroups with $A_1$ elementary, $A_2/A_1$ elementary-free, and $G/A_2$ elementary.
\end{thm}

\section{Locally pro-nilpotent t.d.l.c.s.c. groups}
Our investigations begin with a general discussion of locally pro-nilpotent t.d.l.c.s.c. groups.  

\subsection{Structure theorems}
We take as a convention that discrete groups are locally pro-$\pi$ for any finite set of primes $\pi$.

\begin{prop}\label{prop:loc_nil}
Suppose $G$ is a t.d.l.c.s.c. group that is locally pro-nilpotent. For each prime $p$, there is a closed characteristic subgroup $L_p$ so that 
\begin{enumerate}[(1)]
\item~$\;G/L_p$ is locally pro-$p'$.
\item~There is a countable increasing exhaustion $(H_i)_{i\in \omega}$ of $L_p$ by compactly generated open subgroups so that $H_i/\Rad{\mc{LE}}(H_i)$ is locally pro-$p$.
\end{enumerate}
\end{prop}

\begin{proof}
If $G$ is already locally pro-$p'$, then $L_p:=\{1\}$ satisfies the theorem. Suppose $G$ is not locally pro-$p'$ and consider 
\[
\mc{D}:=\left\{C\in S(G)\mid \forall\; V\in \Uca(G)\;\exists\text{ pro-nilpotent}\; W\sleq_oV:C\sleq\bigcap_{q\in \mb{P}\setminus\{p\} } N_G\left(W_{q}\right)\right\}
\] 
where $\mb{P}$ denotes the set of primes. It is easy to check $\mc{D}$ is invariant and hereditary. We may thus form $N_{\mc{D}}$, the $\mc{D}$-core. Set $L_p:=\ol{N_{\mc{D}}}$ and note $L_p$ is a closed characteristic subgroup. \par

\indent Fix $U\in \Uca(G)$ pro-nilpotent. Letting $(n_j)_{j\in \omega}$ be a countable dense subset of $N_{\mc{D}}$, form the subgroups 
\[
H_j:=\grp{U\cap L,n_0,\dots,n_j}.
\]
Certainly, $(H_j)_{j\in \omega}$ is an $\subseteq$-increasing sequence of compactly generated open subgroups of $L_p$ that exhausts $L_p$. By construction of $H_j$, we infer that $H_j\in \mc{D}$, so there is a pro-nilpotent $W\sleq_oU$ such that $H_j\sleq \bigcap_{q\in \mb{P}\setminus\{p\} } N_G\left(W_{q}\right)$. Therefore, $H_j\cap W_{q}\sleq \Rad{\mc{LE}}(H_j)$ for each $q\in \mb{P}\setminus \{p\}$. Furthermore, by the uniqueness of $W_{q}$ in $W$, we see 
\[
\left(H_j\cap W\right)_{q}\sleq H_j\cap W_{q}
\]
and conclude $H_j/\Rad{\mc{LE}}(H_j)$ is locally pro-$p$. We have thus verified $(2)$.\par

\indent For $(1)$, fixing $C\in \mc{D}$, $u\in U_{p}$, and $V\in \Uca(G)$, there is $W\sleq_oU\cap V$ such that $C\sleq \bigcap_{q\in \mb{P}\setminus\{p\} } N_G\left(W_{q}\right)$. Since $W_{q}\sleq U_{q}$, the element $u$ centralizes $W_{q}$  via Proposition~\rm\ref{prop:pro-nil}. Hence, 
\[
\cgrp{u,C}\sleq \bigcap_{q\in \mb{P}\setminus\{p\} } N_G\left(W_{q}\right).
\]
It follows $u\in N_{\mc{D}}$, and we conclude $U_{p}\sleq N_{\mc{D}}$. The group $G/L_p$ is therefore locally pro-$p'$.
\end{proof}

Proposition~\ref{prop:loc_nil} is the essential tool for proving a surprising decomposition result. 

\begin{thm}\label{thm:loc_nil_main}
Suppose $G$ is a t.d.l.c.s.c. group that is locally pro-nilpotent t.d.l.c.s.c. group. Then either $G$ is elementary or the ascending elementary series
\[
\{1\}\sleq A_1\sleq A_2\sleq G
\] 
is so that
\begin{enumerate}[(1)]
\item~$A_1$ and $G/A_2$ are elementary; and

\item~there is a possibly infinite set of primes $\pi$ so that $A_2/A_1\simeq \bigoplus_{p\in\pi }(L_p,U_p)$ where $L_p$ is a locally pro-$p$ t.d.l.c.s.c. group and $U_p\in \Uca(L_p)$ is pro-$p$.
\end{enumerate}
\end{thm}

\begin{proof}
Suppose $G$ is non-elementary. By passing to $A_2/A_1$, we may assume $G$ is elementary free. Fix $U\in \mathcal{U}(G)$ pro-nilpotent, let $\pi$ list the primes $p$ so that $U$ has a non-trivial $p$-Sylow subgroup, and form the closed characteristic subgroups $L_p$ as given by Proposition~\ref{prop:loc_nil} for each $p\in \pi$.\par

\indent For $p\neq q$ primes from $\pi$, the group $L_p\cap L_q$ is a closed normal subgroup of $G$. Taking $H\sleq L_p\cap L_q$ a compactly generated subgroup, Proposition~\ref{prop:loc_nil} implies $H/\Rad{\mc{LE}}(H)$ is both locally pro-$p$ and locally pro-$q$. The group $H/\Rad{\mc{LE}}(H)$ is therefore discrete, and since $\Rad{\mc{LE}}(H)$ is elementary, the group $H$ is elementary. We conclude $L_p\cap L_q$ is a countable increasing union of elementary subgroups and therefore, is elementary. Since $G$ is elementary-free, it must be the case that $L_p\cap L_q=\{1\}$. \par

\indent In view of the proof of Proposition~\ref{prop:loc_nil}, the unique $p$-Sylow subgroup $U_p$ of $U$ is a subgroup of $L_p$. Hence, $U_p\sleq U\cap L_p$. On the other hand, if $U\cap L_p$ has a non-trivial $q$-Sylow subgroup for $q\neq p$, then the uniqueness of $U_q$ implies $L_p\cap U_q$ is non-trivial contradicting that $L_p\cap L_q$ is trivial. Therefore, $U_p=U\cap L$, and $U_p$ is a compact open subgroup of $L_p$ that is pro-$p$. In particular, $L_p$ is a locally pro-$p$ subgroup of $G$. \par

\indent We now form the t.d.l.c.s.c. group $\bigoplus_{p\in \pi}(L_p,U_p)$. In view of Proposition~\ref{prop:pro-nil}, we may identify $\prod_{p\in \pi}U_p$ with $U$, and this allows us to define $\psi:\bigoplus_{p\in \pi}(L_p,U_p)\rightarrow G$ by 
\[
(l_{p_0},\dots,l_{p_n},u_{p_{n+1}},u_{p_{n+2}},\dots)\mapsto l_{p_0}\dots l_{p_n}\cdot (1,\dots,1,u_{p_{n+1}},u_{p_{n+2}},\dots).
\]
Since $L_p$ centralizes $L_q$ for $p\neq q$ and $U=\prod_{p\in \pi}U_p$, the map $\psi$ is indeed a continuous homomorphism with an open image. The image of $\psi$ equals $\cgrp{L_p\mid p\in \pi}$, so the image is also normal. Since $G$ is elementary-free, we conclude $\psi$ is surjective. \par

\indent We now argue $\psi$ is injective. Suppose for contradiction $(l_{p_0},\dots,l_{p_n},u_{p_{n+1}},\dots)\mapsto 1$ is non-trivial. It follows $M:=\cgrp{L_{p_0}\dots L_{p_n}}\cap \prod_{m>n}U_{p_m}$ is non-trivial. The group $M$ is contained in $\prod_{m>n}U_{p_m}$, hence there is some prime $p_k$ with $k>m$ so that the $p_k$-Sylow subgroup of $M$ is non-trivial. Since the $p_k$-Sylow subgroup, $U_{p_k}$, of $U$ is unique, it follows $M\cap L_{p_k}$ contains the $p_k$-Sylow subgroup of $M$ and, in particular, is non-trivial. \par

\indent On the other hand, $L_{p_k}$ commutes with $\cgrp{L_{p_0}\dots L_{p_n}}$, so $M\cap L_{p_k}$ is a central subgroup of $L_{p_k}$. As $ \Rad{\Es}(L_{p_k})$ contains the centre of $L_{p_k}$, we conclude $\Rad{\Es}(L_{p_k})$ is non-trivial. The subgroup $\Rad{\Es}(L_{p_k})$, however, is a characteristic elementary subgroup of $L_{p_k}$, so $\Rad{\Es}(L_{p_k})$ is a non-trivial elementary normal subgroup of $G$. This contradicts that $G$ is elementary-free. Thus, $\psi$ is injective. \par

\indent We now conclude $\bigoplus_{p\in \pi}(L_p,U_p)\simeq G$ verifying the theorem.
\end{proof}
We remark that when $\pi$ is finite in Theorem~\ref{thm:loc_nil_main}, the local direct product is a direct product.

\begin{cor}\label{cor:loc_nil_simple}
A topologically simple locally pro-nilpotent t.d.l.c.s.c. group $G$ is either elementary or locally pro-$p$ for some prime $p$. 
\end{cor}

The previous corollary implies an interesting theorem of Y. Barnea, M. Ershov, and T. Weigel.

\begin{cor}[Barnea, Ershov, Weigel {\cite[Corollary 4.10]{BEW11}}]\label{cor:BEW}
Suppose $G$ is a t.d.l.c. group that is non-discrete, compactly generated, and topologically simple. If $G$ is locally pro-nilpotent, then $G$ is locally pro-$p$ for some prime $p$.
\end{cor}

\begin{proof}
By \cite[(8.7)]{HR79}, $G$ is second countable, and by \cite[Proposition 6.3]{W_1_14}, $G$ is non-elementary. In view of Corollary~\ref{cor:loc_nil_simple}, we conclude $G$ is locally pro-$p$ for some prime $p$.
\end{proof}

\subsection{An example} 
Fixing $p\neq q$ primes and letting $\widehat{\Zb}_q$ be the $q$-adic integers, we form
\[
P:= PSL_3(\Qi{q})\times \widehat{\Zb}_q\times PSL_3(\Qi{p})
\]
and give $P$ the product topology. The group $P$ is a locally pro-nilpotent and locally pro-$\{p,q\}$ t.d.l.c.s.c. group. \par

\indent We now compute $L_p$ as given by Proposition~\ref{prop:loc_nil}. It is easy to verify $\widehat{\Zb}_q\times PSL_3(\Qi{p})\sleq L_p$. On the other hand,
\[
P/\left(\widehat{\Zb}_q\times PSL_3(\Qi{p})\right)\simeq PSL_3(\Qi{q}).
\]
Letting $\pi$ be the usual projection, $\pi(L_p)$ is then a closed normal subgroup of $PSL_3(\Qi{q})$. The group $PSL_3(\Qi{q})$ is topologically simple, so $\pi(L_p)$ is either trivial or the entire group. Suppose for contraction $\pi(L_p)=PSL_3(\Qi{q})$. Thus, $\pi(N_{\mc{D}})$ is a dense subgroup of $\pi(L_p)$, and moreover, each $g\in \pi(N_{\mc{D}})$ normalizes a compact open subgroup of $PSL_3(\Qi{q})$. Results of Willis and H. Gl\"{o}ckner \cite[Theorem 5.2]{GW01} now imply $\pi(L_p)=PSL_3(\Qi{q})$ is elementary contradicting \cite[Proposition 6.3]{W_1_14}. We conclude $L_p= \widehat{\Zb}_q\times PSL_3(\Qi{p})$. This example shows we \emph{may not assume} $L_p$ is locally pro-$p$ if $P$ is not elementary free.\par

\indent It is now easy to verify the ascending elementary series for $P$ is so that $A_1=\widehat{\Zb}_q$ and $A_2=P$. Hence, $A_2/A_1\simeq PSL_3(\Qi{q})\times PSL_3(\Qi{p})$, which is a direct product of a locally pro-$q$ group with a locally pro-$p$ group. We have thus computed the decomposition given by Theorem~\ref{thm:loc_nil_main}.

\section{Lie groups over the $p$-adics}
We now consider l.c.s.c. $p$-adic Lie groups. The primary references for this section are \cite{Bo98} for Lie group theory and \cite{Mar91} for algebraic group theory.

\subsection{Preliminaries} 
\begin{defn}
A \textbf{Lie group} is a topological group with a $\mathbb{K}$-manifold structure such that the group operations are analytic where $\mathbb{K}$ is either $\mathbb{R}$, $\mathbb{C}$, or some non-discrete complete ultrametric field.
\end{defn}

\indent A \textbf{$p$-adic Lie group} is a Lie group over $\mathbb{Q}_p$, the $p$-adic numbers; note that l.c.s.c. $p$-adic Lie groups are t.d.l.c.s.c. groups. For our purposes, the following characterization of l.c. $p$-adic Lie groups is much more useful:

\begin{thm}[Lazard \cite{L65}, Lubotzky, Mann \cite{LM87}]\label{thm:p-adic_char}
Suppose $G$ is a t.d.l.c. group. Then $G$ is a $p$-adic Lie group if and only if $G$ has a compact open subgroup that is pro-$p$ and has finite rank.
\end{thm}

\indent A l.c.s.c. $p$-adic Lie group $G$ comes with a Lie algebra, denoted $L(G)$. There is a canonical representation $Ad:G\rightarrow \Aut(L(G))$ with $\Aut(L(G))$ a closed subgroup of $GL_n(\mathbb{Q}_p)$ for some $n$, \cite[III.3.11 Corollary 5]{Bo98}. \par

\indent We also make use of algebraic group theory; in the following $k$ denotes a field.

\begin{defn}
An \emph{algebraic group} is an algebraic variety $G$ with group operations given by $\cdot$ and $*^{-1}$ such that $\cdot: G^{\times 2}\rightarrow G$ and $*^{-1}:G\rightarrow G$ are morphisms of algebraic varieties. 
\end{defn}

\indent An algebraic group $G$ is \textbf{connected} if it is connected in the Zariski topology. A subgroup $H\sleq G$ is \textbf{algebraic} if $H$ is an algebraic subvariety of $G$. When an algebraic group is a variety defined over a field $k$, we say $G$ is a \textbf{$k$-algebraic group}. In general, the prefix ``$k$-" is applied to indicate an algebraic object is defined over the field $k$.\par

\indent There are a number of notions of simplicity in the category of algebraic groups. Since confusing these definitions is easy and dangerous, we enumerate these: An algebraic group $G$ is called \textbf{absolutely simple} if $\{1\}$ is the only proper algebraic normal subgroup of $G$. If all proper algebraic normal subgroups are finite, $G$ is called \textbf{almost absolutely simple}. If $\{1\}$ is the only proper $k$-algebraic normal subgroup, $G$ is called $k$\textbf{-simple}. If all proper $k$-algebraic normal subgroups are finite, $G$ is said to be \textbf{almost $k$-simple}.\par

\indent For $G$ a $k$-algebraic group, the set of $k$-rational points of $G$ is denoted $G(k)$; the $k$-rational points are the solutions in $k^{\times n}$ to the polynomials defining $G$. The set $G(k)$ becomes a group under the group operations of $G$. There is a canonical normal subgroup of $G(k)$, denoted $G(k)^+$. Under the assumption $k$ has characteristic zero, $G(k)^+$ is the subgroup generated by all unipotent elements of $G(k)$. (The subgroup $G(k)^+$ exists in the prime characteristic case; the definition, however, is a bit more complicated.) When $k$ is a \emph{local field}, i.e. a non-discrete locally compact topological field, the set of $k$-rational points inherits a topology from the local field $k$. This topology makes $G(k)$ a locally compact, $\sigma$-compact, and metrizable topological group. Whenever we say ``closed in $G(k)$", we mean closed with respect to this topology.\par

\indent We are now prepared to give a central definition:

\begin{defn}[\cite{CCLTV11}]
A $p$-adic Lie group $G$ is said to be of \textbf{simple type} if $G\simeq S(\Qp)^+$ where $S$ is an almost $\Qp$-simple isotropic $\Qp$-algebraic group. We say $G$ is of \textbf{adjoint simple type} if $S$ is also adjoint.
\end{defn}
Note adjoint almost $k$-simple $k$-algebraic groups are $k$-simple.

\subsection{Topologically simple l.c.s.c. $p$-adic Lie groups}
We now show non-elementary topologically simple l.c.s.c. $p$-adic Lie groups are isomorphic to closed subgroups of $GL_n(\Qp)$ and have simple Lie algebras. This is not immediate in the $p$-adic setting because $Ad$ is not necessarily a closed map and there is no longer a bijective correspondence between ideals of the Lie algebra and normal subgroups of the Lie group. 

\begin{prop}\label{prop:sgrp_gl}
If $G$ is a non-elementary topologically simple l.c.s.c. $p$-adic Lie group, then $G$ is isomorphic to a closed subgroup of $GL_n(\Qp)$ and $L(G)$ is simple.
\end{prop}
\begin{proof}
Let $\mathfrak{r}$ be the solvable radical of $L(G)$ and suppose for contradiction $\mathfrak{r}$ is non-trivial. Via the Levi-Malcev theorem, see \cite[I.6.8 Theorem 5]{Bo98}, $\mathfrak{r}$ has a supplement. There is thus $O\sleq_o G$ and a non-trivial $N\trianglelefteq O$ such that $L(N)=\mathfrak{r}$; cf. \cite[ III.7.1 Proposition 2]{Bo98}. Since $\mf{r}$ is solvable, there is $W\sleq_oN$ that is solvable. We may find $U\in \Uca(G)$ with $U\sleq O$ and $U\cap N\sleq W$. The penultimate term in the derived series of $N\cap U$ is thus a non-trivial locally normal abelian subgroup of $G$ contradicting Theorem~\rm\ref{thm:[A]semisimple}. We conclude $\mathfrak{r}$ is trivial and, therefore, $L(G)$ is semisimple.\par

\indent Observe the map $Ad$ is injective. Indeed, else $G=\ker(Ad)$, and applying \cite[Proposition 3.1]{GW01}, every element of $G$ normalizes a compact open subgroup - such a $G$ is called \textbf{uniscalar}. Writing $G=\bigcup_{i\in \omega}O_i$ with $(O_i)_{i\in \omega}$ an $\subseteq$-increasing sequence of compactly generated open subgroups of $G$, each $O_i$ is also uniscalar and via \cite[Theorem 5.2]{GW01}, elementary. This contradicts that $G$ is non-elementary.\par

\indent Fix $U\in \Uca(G)$. In view of \cite[III.3.8 Proposition 28]{Bo98} and the previous paragraph, $U\simeq Ad(U)$ as Lie groups, and $L(G)=L(U)=L(Ad(U))$. Since $L(G)$ is semisimple, \cite[LA 6.7 Corollary 2]{Se64} and \cite[III.3.11 Corollary 5]{Bo98} imply $L(G)=L(\Aut(L(G)))$. Via \cite[III.7.1 Theorem 2]{Bo98}, $Ad(U)$ is then open in $\Aut(L(G))$, and it follows $G\simeq Ad(G)$ and $Ad(G)$ is closed. Since $\Aut(L(G))$ is a closed subgroup of $GL_n(\Qp)$, we conclude $G$ is isomorphic to a closed subgroup of $GL_n(\Qp)$.\par

\indent To see the simplicity of $L(G)$, say $L(G)=\mf{s}_1\times\dots\times \mf{s}_n$ with the $\mf{s}_i$ simple Lie algebras and let $W\sleq \Aut(L(G))$ be the collection of $g\in \Aut(L(G))$ such that $g(\mf{s}_i)=\mf{s}_i$ for each $1\sleq i\sleq n$. The subgroup $W$ is finite index in $\Aut(L(G))$, so $W\cap Ad(G)=Ad(G)$. Identifying $W$ with $\Aut(\mf{s}_1)\times \dots \times \Aut(\mf{s}_n)$, we consider $Ad(G)\cap \Aut(\mf{s}_i)$. If $Ad(G)\cap \Aut(\mf{s}_i)=\{1\}$ for each $i$, then $\Aut(\mf{s}_i)$ is discrete for each $i$ which is absurd. Therefore, $Ad(G)\sleq \Aut(\mf{s}_i)$ for some $i$, and since $Ad(G)$ is open in $\Aut(L(G))$, we infer
\[
L(G)=L(Ad(G))=L(\Aut(\mf{s}_i))=\mf{s}_i,
\]
so $L(G)$ is simple.
\end{proof}

\begin{rmk} 
The above argument can easily be adapted to give a proof of a result of Caprace, Reid, and Willis in forthcoming work \cite{CRW13}: An $[A]$-semisimple l.c.s.c. $p$-adic Lie group has a semisimple Lie algebra.
\end{rmk}

We now show non-elementary topologically simple l.c.s.c. $p$-adic Lie groups have a close relationship with $\Qp$-algebraic groups. This requires a result from the literature.

\begin{prop}[{\cite[Proposition 6.5]{CCLTV11}}]\label{prop:char_simple}
Let $G$ be a non-compact closed subgroup of $GL_n(\Qp)$ such that every
proper, closed, characteristic subgroup is compact. Then one of the following cases holds:
\begin{enumerate}[(1)]
\item~The closed subgroup generated by all solvable normal subgroups of $G$ is compact open, central in $G$.

\item~$G$ is isomorphic to $\Qp^{\times k}$ for some $k>0$.

\item~$G\simeq S^{\times k}/Z$ where $S$ is a $p$-adic Lie group of simple type, $k\sgeq 1$, and $Z$ is a
central subgroup of $S^{\times k}$ that is invariant under a transitive group of permutations of $\{1,\dots, k\}$.
\end{enumerate}
\end{prop} 

With our preliminary work in hand, we now prove the first theorem of this section.

\begin{thm}\label{thm:simple_type}
Suppose $G$ is a non-elementary topologically simple l.c.s.c. $p$-adic Lie group. Then $G\simeq H(\Qp)^+$ with $H$ an adjoint $\Qp$-simple isotropic $\Qp$-algebraic group. That is to say, $G$ is of adjoint simple type.
\end{thm}

\begin{proof}
Via Proposition~\rm\ref{prop:sgrp_gl}, we may identify $G$ with a closed subgroup of $GL_n(\Qp)$ and apply Proposition~\ref{prop:char_simple}. Cases $(1)$ and $(2)$ are impossible since $G$ is non-elementary, so $G\simeq S^{\times k}/Z$ for some $k\sgeq 1$. Suppose $k>1$ for contradiction and let $\psi:S^{\times k}\rightarrow G$ be the obvious homomorphism. Writing $S^{\times k}=S_1\times\dots\times S_k$, we see $\psi(S_i)\trianglelefteq G$ for each $i$, hence $\ol{\psi(S_i)}=G$ for some $i$. We may assume $i=1$.\par

\indent Taking $i\neq 1$, the group $\psi(S_i)$ commutes with $\psi(S_1)$. It follows $G$ commutes with $\psi(S_i)$, hence $\psi(S_i)$ is central in $G$. Since $G$ is non-elementary, $\psi(S_i)=\{1\}$, and $S_i\sleq Z$. The group $S$ is therefore abelian, and this is absurd as $G$ is non-elementary. We conclude $G\simeq S/Z$ where $S=T(\Qp)^+$ for $T$ some almost $\Qp$-simple isotropic $\Qp$-algebraic group. \par

\indent Let $\ol{T}$ be the adjoint group of $T$ and let $\overline{p}:T\rightarrow \ol{T}$ be the canonical $\Qp$-isogeny, \cite[1.4.11]{Mar91}. Since $T$ is almost $\Qp$-simple and isotropic, $\ol{T}$ is an adjoint $\Qp$-simple isotropic $\Qp$-algebraic group. Further, $\ol{p}:T(\Qp)^+\rightarrow \ol{T}(\Qp)^+$ continuously and surjectively, \cite[2.3.4]{Mar91} and \cite[1.5.5]{Mar91}. By \cite[1.5.6]{Mar91}, the group $\ol{T}(\Qp)^+$ has trivial centre, whereby 
\[
\ker(\ol{p}\upharpoonright_{T(\Qp)^+})=Z.
\]
It follows $G\simeq \ol{T}(\Qp)^+$, and we conclude $G$ is of adjoint simple type.
\end{proof}

In the next section, $\Qp$-simplicity is not sufficient for our purposes. There is, fortunately, a way to account for this.
\begin{fact}[{\cite[1.7]{Mar91}}] Let $k$ be a local field. If $S$ is an adjoint $k$-simple isotropic $k$-algebraic group, then there is $k'$ a finite separable extension of $k$ and an adjoint absolutely simple isotropic $k'$-algebraic group $S'$ such that $S'(k')^+\simeq S(k)^+$ as topological groups.
\end{fact}

\begin{cor}\label{cor:simple_type}
If $G$ is a non-elementary topologically simple l.c.s.c. $p$-adic Lie group, then $G\simeq H(k)^+$ as topological groups where $H$ is an adjoint absolutely simple isotropic $k$-algebraic group with $k$ a finite extension of $\Qp$.
\end{cor}

\subsection{Automorphism groups}
For $k$ a local field and $S$ an algebraic group, we consider $\Aut\left(S(k)^+\right)$ to be the group of topological group automorphisms of $S(k)^+$, $\Inn\left(S(k)^+\right)$ to be the collection of inner automorphisms, $\Aut(S)$ to be the group of algebraic automorphisms of $S$, and $\Aut(k)$ to be the group of field automorphisms of $k$. We do not consider the elements of $\Aut(k)$ to be continuous.\par

\indent We now study the automorphism groups of non-elementary topologically simple l.c.s.c. $p$-adic Lie groups. In view of the previous subsection, this is really a question concerning the automorphism group of $S(k)^+$ for a $k$-algebraic group $S$. Conveniently, there is a deep theorem concerning such automorphism groups:  

\begin{thm}[{\cite[Corollaire 8.13]{BT73}}, cf. {\cite[1.8.2]{Mar91}}]\label{thm:BT}
Suppose $k$ is a finite extension of $\Qp$ and $S$ is a $k$-algebraic group that is adjoint, absolutely simple, and isotropic. If $\alpha\in \Aut\left(S(k)^+\right)$, then there is a unique $\phi\in \Aut(k)$ and unique $k$-isomorphism $\beta: \prescript{\phi}{}{S}\rightarrow S$ such that $\alpha=\beta\circ \phi^0\upharpoonright_{S(k)^+}$.
\end{thm}
\noindent A definition of $\prescript{\phi}{}{S}$ and $\phi^0$ may be found in \cite[1.7]{Mar91}.\par

\indent Let us explore Theorem~\rm\ref{thm:BT} further. Suppose $k$ is a finite extension of $\Qp$ and $S$ is an adjoint absolutely simple isotropic $k$-algebraic group. For each $\alpha \in \Aut\left(S(k)^+\right)$, there is $\phi_{\alpha}\in \Aut(k)$ and $\beta_{\alpha}:\prescript{\phi_{\alpha}}{}{S}\rightarrow S$ as given by Theorem~\rm\ref{thm:BT}. For each $\phi \in \Aut(k)$, we may thus define 
\[
\Sigma_{\phi}:=\left\{\alpha \in \Aut\left(S(k)^+\right)\mid\phi_{\alpha}=\phi\right\},
\]
so 
\[
\Aut\left(S(k)^+\right)=\bigcup_{\phi\in \Aut(k)}\Sigma_{\phi}.
\]
Via \cite[2.3]{BT73}, each $\phi\in \Aut(k)$ is continuous and, therefore, is the identity on $\Qp$. We infer $\Aut(k)=\Aut(k/\Qp)$ is indeed a finite group. It follows additionally $\Inn\left(S(k)^+\right)\sleq \Sigma_1$; cf. Theorem~\rm\ref{thm:fin_out}.

\begin{claim}\label{cl:sigma}
$\Sigma_1$ is a finite index subgroup of $\Aut(S(k)^+)$.
\end{claim}

\begin{proof}
Let us first check $\Sigma_1$ is a subgroup. Take $\alpha_1,\alpha_2\in \Aut\left(S(k)^+\right)$ and let $\beta_1$ and $\beta_2$ be the respective $k$-isomorphisms given by Theorem~\rm\ref{thm:BT}. The map $\beta_1\circ \beta_2:G\rightarrow G$ is again a $k$-isomorphism, and
\[
\beta_1\circ \beta_2\upharpoonright_{S(k)^+}=\alpha_1\circ \alpha_2.
\]
By the uniqueness of $\beta_{\alpha_1\circ \alpha_2}$ and $\phi_{\alpha_1\circ \alpha_2}$, we conclude $\beta_{\alpha_1\circ \alpha_2}= \beta_1\circ \beta_2$ and $\phi_{\alpha_1\circ \alpha_2}=1$, whereby $\Sigma_1$ is closed under composition. That $\Sigma_1$ is closed under taking inverses follows similarly.\par

\indent To see $\Sigma_1$ has finite index, it suffices to show for each $\phi \in \Aut(k)$, $\Sigma_{\phi}$ is contained in a right coset of $\Sigma_1$. To this end, fix $\phi \in \Aut(k)$, fix $\alpha_1\in\Sigma_{\phi}$, and let $\beta_1:\prescript{\phi}{}{S}\rightarrow S$ be the $k$-isomorphism given by Theorem~\ref{thm:BT} for $\alpha_1$. Consider $\alpha_2\in\Sigma_{\phi}$ and let $\beta_2:\prescript{\phi}{}{S}\rightarrow S$ be the $k$-isomorphism given by Theorem~\rm\ref{thm:BT} for $\alpha_2$. The map $\beta_{2}\circ \beta_{1}^{-1}$ is then a $k$-automorphism of $S$, and for $x\in S(k)^+$,
\[
\beta_{2}(\beta_{1}^{-1}(x))=\beta_{2}(\phi^0(\alpha^{-1}_1(x)))=\alpha_2\circ \alpha^{-1}_1(x).
\]
By the uniqueness of $\beta_{\alpha_2\circ \alpha^{-1}_1}$ and $\phi_{\alpha_2\circ \alpha^{-1}_1}$, we conclude $\beta_{\alpha_2\circ \alpha^{-1}_1}=\beta_{2}\circ \beta_{1}^{-1}$ and $\phi_{\alpha_2\circ \alpha^{-1}_1}=1$. Therefore, $\alpha_2\circ \alpha_1^{-1}\in \Sigma_1$, so $\alpha_2\in \Sigma_1\alpha_1$. The set $\Sigma_{\phi}$ is thus contained in $\Sigma_1\alpha_1$ verifying the claim.
\end{proof}
As a corollary to the proof of Claim~\rm\ref{cl:sigma}, the map $\Phi:\Sigma_1 \rightarrow \Aut(S)$ defined by $\alpha \mapsto \beta_{\alpha}$ is an injective group homomorphism. \par

\indent To prove the second theorem of this section, we require two additional facts concerning $\Aut(S)$.

\begin{thm}[{\cite[5.7.2]{Ti74}} ]\label{thm:aut_alg} Let $G$ be a semisimple $k$-algebraic group. Then there exists a $k$-algebraic group $\Aut(G)$ such that for any field extension $l$ of $k$, the group $\Aut(G)(l)$ is the group of $l$-automorphisms of $G$. 
\end{thm}

\begin{thm}[{\cite[5.7.2]{Ti74}}]\label{thm:fin_out} If $G$ is an adjoint semisimple $k$-algebraic group, then $G$ is $k$-isomorphic to $\Aut(G)^{\circ}$ where $\Aut(G)^{\circ}$ is the connected component of $1$ and the isomorphism is given by the self action of $G$ by conjugation.
\end{thm}

We are now equipped to prove the desired theorem.
\begin{thm}\label{thm:index} 
Suppose $G$ is a non-elementary topologically simple l.c.s.c. p-adic Lie group, let $\Aut(G)$ be the group of topological group automorphisms of $G$, and let $\Inn(G)$ be the collection of inner automorphisms. Then $\Aut(G)/\Inn(G)$ is finite.
\end{thm}

\begin{proof}
In view of Corollary~\rm\ref{cor:simple_type}, we may take $G=S(k)^+$ where $k$ is a finite extension of $\Qp$ and $S$ is an adjoint absolutely simple isotropic $k$-algebraic group. \par

\indent Let $\xi: S\rightarrow \Aut(S)^{\circ}$ send $s\in S$ to the algebraic automorphism given by conjugation with $s$. By Theorem~\rm\ref{thm:fin_out}, $\xi$ is a $k$-isomorphism of $k$-algebraic groups, so $\xi(S(k))= \Aut(S)^{\circ}(k)$; cf. \cite[2.1.2]{Mar91}. On the other hand, let $\Phi:\Sigma_1\rightarrow \Aut(S)$ be as described above and identify 
\[
S(k)^+\simeq \Inn(S(k)^+)\sleq \Aut(S(k)^+).
\]
For $x\in S(k)^+$, we then see $\xi(x)\upharpoonright_{S(k)^+}=\Phi(x)\upharpoonright_{S(k)^+}$, and since $\xi(x)$ is a $k$-automorphism of $S$, the uniqueness of $\Phi(x)$ implies $\xi(x)=\Phi(x)$. Therefore,
\[
\Phi^{-1}\left(\xi(S(k)^+)\right)=\Inn\left(S(k)^+\right).
\]

\indent Now the group $\Phi^{-1}(\Aut(S)^{\circ})$ is finite index in $\Sigma_1$, and 
\[
\Phi:\Phi^{-1}\left(\Aut(S)^{\circ}\right)\rightarrow \Aut(S)^{\circ}(k)=\xi\left(S(k)\right).
\]
Applying {\cite[2.3.1]{Mar91}}, $\xi\left(S(k)^+\right)$ is finite index in $\xi\left(S(k)\right)$, so 
\[
\Phi^{-1}\left(\xi(S(k)^+)\right)=\Inn\left(S(k)^+\right)
\]
is finite index in $\Phi^{-1}(\Aut(S)^{\circ})$. We infer $\Inn(S(k)^+)$ is finite index in $\Sigma_1$. Claim~\rm\ref{cl:sigma} implies $\Inn(S(k)^+)$ is indeed finite index in $\Aut(S(k)^+)$ verifying the theorem.
\end{proof}

\section{Locally of finite rank t.d.l.c.s.c. groups}
We now bring together our previous results to prove the main theorems of this work. We will require an easy preliminary lemma.

\begin{lem}\label{lem:fitt_comm}
If $G$ is a t.d.l.c.s.c. group and $U,V\in \Uca(G)$, then $F(U)\sim_cF(V)$ where $F(U)$ and $F(V)$ are the Fitting subgroups of $U$ and $V$, respectively. In particular, $F(U)$ is commensurated by $G$. 
\end{lem}

\begin{proof}
Fix $U\in \Uca(G)$ and suppose first $V\trianglelefteq_oU$. Plainly, $ F(U)\cap V = F(V)\sleq F(U)$, so $F(V)\sim_c F(U)$. For $V\in \Uca(G)$ arbitrary, take $W\trianglelefteq_oU$ such that $W\sleq U\cap V$. By our first observation, $F(U)\sim_c F(W)$. On the other hand, $W\trianglelefteq_o U\cap V$, so $F(W)\sim_cF(U\cap V)$. Hence, $F(U)\sim_cF(U\cap V)$. Reversing the roles of $U$ and $V$, we conclude $F(U\cap V)\sim_cF(V)$, and the lemma follows.
\end{proof}

\subsection{Elementary-free Groups}

When considering a t.d.l.c.s.c. group, we may always reduce to the elementary-free normal section $A_2/A_1$ via Theorem~\rm\ref{thm:acs_E_series}. In view of this, we begin by examining locally of finite rank groups that are elementary-free. 

\begin{lem}\label{lem:pi-adic}
Suppose $G$ is a t.d.l.c.s.c. group that is locally of finite rank and locally pro-$\pi$ for some finite set of primes $\pi$. If $G$ is non-elementary and topologically simple, then $G$ is a non-elementary topologically simple l.c.s.c. $p$-adic Lie group for some $p\in \pi$
\end{lem}

\begin{proof} 
Via Corollary~\ref{cor:melnikov}, $G$ is locally pro-nilpotent, and since $G$ is non-elementary, Corollary~\ref{cor:loc_nil_simple} implies $G$ is locally pro-$p$. Applying Theorem~\rm\ref{thm:p-adic_char}, $G$ is a l.c.s.c. $p$-adic Lie group.
\end{proof}

\indent We now find minimal non-trivial normal subgroups and show they are non-elementary topologically simple $p$-adic Lie groups. 

\begin{lem}\label{lem:minlem}
Suppose $G$ is a non-trivial t.d.l.c.s.c. group that is locally of finite rank and locally pro-$\pi$ for some finite set of primes $\pi$. If $G$ is elementary-free, then $G$ has exactly $0<k<\infty$ many minimal non-trivial normal subgroups.
\end{lem}

\begin{proof} Put $\mc{N}:=\{N\trianglelefteq G\mid\{1\}\neq N\}$. We claim filtering families in $\mc{N}$ have non-trivial intersection; recall a filtering family of subgroups is such that for any two members of the family there is a third member contained in both. Suppose $(N_{\alpha})_{\alpha\in I}$ is a filtering family in $\mc{N}$ and put $N:=\bigcap_{\alpha\in I}N_{\alpha}$. For contradiction, suppose $N=\{1\}$.\par

\indent Fix $U\in \Uca(G)$ such that $U$ has finite rank and is pro-$\pi$. Via Corollary~\ref{cor:melnikov}, we may take $U$ to be pro-nilpotent. Let $(g_n)_{n\in \omega}$ list a countable dense subset of $G$ and for each $n\sgeq 0$, put $P_n:=\grp{U,g_0,\dots,g_n}$. By \cite[Proposition 2.5]{CM11}, the normal core of $U$ in $P_n$, $K_n:=\bigcap_{h\in P_n} hUh^{-1}$, is such that every filtering family of non-discrete closed normal subgroups of $P_n/K_n$ has non-trivial intersection. All $K_n$ must then be non-trivial. Indeed, else fix $n$ such that $K_n$ is trivial. The family $(N_{\alpha}\cap P_n)_{\alpha\in I}$ is a filtering family of closed normal subgroups of $P_n$ with trivial intersection, hence some $N_{\alpha}\cap P_n$ is discrete. However, this implies $N_{\alpha}$ is discrete contradicting that $G$ is elementary-free. \par

\indent The action of $P_n$ on $K_n$ by conjugation induces a continuous homomorphism $P_n\rightarrow \Aut(K_n)$ when $\Aut(K_n)$ is given the congruence subgroup topology. The group $K_n$ is topologically finitely generated, so $\Aut(K_n)$ is profinite. Theorem~\rm\ref{thm:closure_main} thus implies $P_n/C_{P_n}(K_n)$ is elementary. Furthermore, since $(K_{n})_{n\in \omega}$ is $\subseteq$-decreasing, $(C_{P_n}(K_n))_{n\in \omega}$ is an $\subseteq$-increasing sequence of subgroups, so 
\[
R:=\ol{\bigcup_{n\in \omega}C_{P_n}(K_n)}
\]
is a normal subgroup of $G$. Certainly, $G/R=\bigcup_{n\in \omega} P_nR/R$, and as $P_nR/R\simeq P_n/P_n\cap R$ is a quotient of $P_n/C_{P_n}(K_n)$, Theorem~\rm\ref{thm:closure_main} implies $P_n/P_n\cap R$ is elementary. The group $G/R$ is therefore elementary, and as $G$ is elementary-free, $\bigcup_{n\in \omega}C_{P_n}(K_n)$ is dense in $G$. \par

\indent Since $\bigcup_{n\in \omega}C_{P_n}(K_n)$ is dense and $\Phi(U)$ is open by Proposition~\rm\ref{prop:s_solvable}, there are $g_1,\dots,g_k$ elements of $C_{P_n}(K_n)$ for some sufficiently large $n$ such that $\cgrp{g_0,\dots, g_k} \Phi(U)=U$. We conclude $\cgrp{g_1,\dots,g_k}=U$, and therefore, $C_{P_n}(K_n)$ is open for all sufficiently large $n$. Each element of $K_n$ for sufficiently large $n$ thus has an open centralizer. This, however, contradicts that $G$ is elementary-free since $QZ(G)\sleq \Rad{\Es}(G)$ and $K_n$ is non-trivial. We conclude $\bigcap_{\alpha\in I}N_{\alpha}$ is non-trivial. \par

\indent As a particular case of the above, decreasing $\subseteq$-chains in $\mc{N}$ have non-trivial intersection. The collection $\mc{N}$ thus admits $\subseteq$-minimal elements by Zorn's lemma. It now remains to show there are finitely many minimal non-trivial normal subgroups. \par

\indent Let $\mc{M}$ list all $\subseteq$-minimal elements of $\mc{N}$ and for each $\mc{E}\subseteq\mc{M}$, let $N_{\mc{E}}:=\cgrp{N\mid N\in \mc{E}}$. Taking $\mc{E}\subsetneq \mc{M}$ and $N\in \mc{M}\setminus \mc{E}$, the group $N_{\mc{E}}$ avoids the non-trivial elements in $N$. Indeed, since $N$ commutes with the generators of $N_{\mc{E}}$, $N_{\mc{E}}\cap N=:M$ is central in $N_{\mc{E}}$. Abelian groups are elementary, whereby $M\sleq \Rad{\Es}(N_{\mc{E}})$. The group $\Rad{\Es}(N_{\mc{E}})$ is characteristic in $N_{\mc{E}}$, so
\[
\Rad{\Es}(N_{\mc{E}})\sleq \Rad{\Es}(G)=\{1\}.
\]
Hence, $M=\{1\}$ as required. \par

\indent Consider $\mc{K}:=\{N_{\mc{M}\setminus \mc{F}}\mid\mc{F}\subseteq \mc{M}\text{ is finite}\}$ where, if needed, $N_{\emptyset}:=\{1\}$. The collection $\mc{K}$ is a filtering family of closed normal subgroups, and supposing for contradiction that $\mc{M}$ is infinite, every element of $\mc{K}$ is non-trivial. Thus, $\mc{K}\subseteq \mc{N}$, and since filtering families in $\mc{N}$ have lower bounds in $\mc{N}$, $K:=\bigcap\mc{K}$ is non-trivial. The subgroup $K$ contains a minimal non-trivial normal subgroup of $G$; fix such a subgroup $L$. The construction of $K$ now implies $ K\trianglelefteq N_{\mc{M}\setminus\{L\}}$, and this is absurd since $N_{\mc{M}\setminus\{L\}}$ avoids $L\setminus\{1\}\subseteq K$. We conclude $\mc{M}$ is finite proving the lemma.
\end{proof}

\indent For $G$ as in the above lemma, we denote the collection of minimal non-trivial normal subgroups by $\mc{M}(G)$. 

\begin{lem}\label{lem:simplelem}
Suppose $G$ is a non-trivial t.d.l.c.s.c. group that is locally of finite rank and locally pro-$\pi$ for some finite set of primes $\pi$. If $G$ is elementary-free, then each $N\in \mc{M}(G)$ is a non-elementary topologically simple $p$-adic Lie group for some $p\in \pi$.
\end{lem}

\begin{proof}
Fix $N\in \mc{M}(G)$. Since the elementary radical and residual of $N$ are characteristic subgroups, it follows $\Rad{\Es}(N)=\{1\}$ and $\Res{\Es}(N)=N$, so $N$ is elementary-free. \par

\indent Via Lemma~\rm\ref{lem:minlem}, $N$ has finitely many non-trivial minimal normal subgroups $K_1,\dots,K_n$. The action of $G$ by conjugation permutes $\{K_1,\dots,K_n\}$ and thereby induces a continuous homomorphism $\psi:G\rightarrow Sym(n)$. Since $\ker(\psi)$ has finite index and $\Res{\Es}(G)=G$, $\psi$ is the trivial homomorphism. Each of $K_1,\dots,K_n$ is then normal in $G$, and the minimality of $N$ implies $K_1=\dots=K_n=N$. Therefore, $N$ is topologically simple.\par

\indent The group $N$ is thus a locally of finite rank, locally pro-$\pi$ t.d.l.c.s.c. group that is elementary free and topologically simple. Applying Lemma~\rm\ref{lem:pi-adic}, $N$ is a l.c.s.c. $p$-adic Lie group for some $p\in \pi$.
\end{proof}

\indent We now consider groups that are not necessarily locally pro-$\pi$ for $\pi$ a finite set of primes. Here we must work harder.
\begin{lem}\label{lem:inf_lfr}
Suppose $U$ is a profinite group with finite rank. Suppose further $U$ has no non-trivial locally normal abelian subgroups and has trivial quasi-centre. If $M\trianglelefteq U$ is pro-$\pi$ for some finite set of primes $\pi$ and $U/M$ is virtually solvable, then $U$ is pro-$\xi$ for some finite set of primes $\xi\supseteq \pi$. 
\end{lem}

\begin{proof}
Let $p_1,\dots,p_n$ list $\pi$. Since $U$ has finite rank, every $p_i$-Sylow subgroup of $U$ is topologically finitely generated. Theorem~\rm\ref{thm:reidcore} thereby implies $U/O_{p'_i}(U)$ is virtually pro-$p_i$ for each $i$.  We may thus find $N_i\trianglelefteq_oU$ such that $N_i$ has a normal subgroup $K_i$ that is pro-$p_i'$ and $N_i/K_i$ is pro-$p_i$ - a so called normal \emph{$p_i'$-Hall subgroup} - for each $1\sleq i \sleq n$.\par

\indent Form $L:=\bigcap_{i=1}^nN_i$.  The group $L$ has a normal $p_i'$-Hall subgroup, $L_{p_i'}$, for each $1\sleq i \sleq n$ since $L\sleq N_i$. Setting $K:=\bigcap_{i=1}^nL_{p_i'}$, it must be the case $M\cap K=\{1\}$ as $K$ is a pro-$\pi'$ subgroup. Hence, $KM/M\simeq K/K\cap M\simeq K$, and $K$ is virtually solvable because $LM/M$ is virtually solvable. \par

\indent If $K$ is finite, then $K$ must be trivial since $U$ has trivial quasi-centre. If $K$ is infinite, then we may find $W\trianglelefteq_oL$ such that $W\cap K$ is solvable. Since $W\cap K\trianglelefteq L$, it follows $L$ has a non-trivial abelian normal subgroup contradicting our assumption on $U$. We conclude $K$ is trivial, so $L/K=L$ is pro-$\pi$ proving the lemma.
\end{proof}

\begin{lem}\label{lem:loc_finrank}
Suppose $G$ is a t.d.l.c.s.c. group and $U\in \Uca(G)$ has finite rank. If $G$ is elementary-free, then for each prime $p$, there is a closed characteristic subgroup $N_p$ so that $F(U)_p$ is a compact open subgroup of $N_p$. Additionally, for primes $p\neq q$, $N_p\cap N_q=\{1\}$.
\end{lem}

\begin{proof}
Fix a prime $p$. Define $\mc{H}_p\subseteq S(G)$ by $C\in \mc{H}_p$ if and only if for all $L\in \Uca(G)$ there is $W\sleq_oL$ and $(W_i)_{i\in \omega}$ a normal basis at $1$ for $W$ such that 
\[
C\sleq \bigcap_{i\in \omega}\bigcap_{q\in \mb{P}\setminus\{p\}}N_G\left(F(W)_{q}\cap W_i\right)
\]
where $F(W)$ is the Fitting subgroup of $W$. The set $\mc{H}_p$ is hereditary, and it is easy to check $\mc{H}_p$ is also conjugation invariant. We may thus form the $\mc{H}_p$-core $N_{\mc{H}_p}=:M_p$.\par

\indent Let $\{h_i\}_{i\in \omega}$ list a countable dense subset of $M_p\cap M_{r}$ for $p\neq r$ primes. Putting $P_i:=\grp{U,h_0,\dots,h_i}$, it follows from the construction of $N_p$ that $P_i\in \mc{H}_p$, so there is $W\sleq_oU$ for which 
\[
P_i \sleq \bigcap_{q\in \mb{P}\setminus \{p\}}N_G\left(F(W)_{q}\right).
\]
Therefore, $F(W)_{q}\sleq \Rad{\Es}(P_i)$ for each $q\neq p$. Letting $\pi:P_i\rightarrow P_i/\Rad{\Es}(P_i)$ be the usual projection, $\pi(F(W))$ is pro-$p$.\par

\indent On the other hand, $P_i\in \mc{H}_r$, whereby we may find $V\sleq_oU$ so that $\pi(F(V))$ is pro-$r$. The group $\pi(F(V))\cap \pi(F(W))$ is then both pro-$p$ and pro-$r$ and thereby is trivial. Lemma~\ref{lem:fitt_comm}, however, tells us that $\pi(F(V))\sim_c\pi(F(W))$. The subgroup $\pi(F(W)$ is thus finite, and we infer $\pi(F(W))$ is trivial because $\pi(F(W))\trianglelefteq \pi(W)$ and $P_i/\Rad{\Es}(P_i)$ has trivial quasi-centre. In view of Theorem~\ref{thm:finrank}, $\pi(W)$ is virtually solvable, and applying Theorem~\ref{thm:locsolv_e}, we conclude $P_i/\Rad{\Es}(P_i)$ is elementary. The maximality of the elementary radical implies $P_i$ is elementary.\par

\indent We see $M_p\cap M_r\sleq \bigcup_{i\in \omega}P_i$, and the latter group is elementary. The group $\ol{M_p\cap M_r}$ is therefore an elementary normal subgroup of $G$ via Theorem~\ref{thm:closure_main}. Since we assumed $G$ is elementary free, it must be the case that $M_p\cap M_r=\{1\}$. So $M_p\sleq C_{G}(M_r)$, and therefore, $N_p:=\ol{M_p}$ centralizes $N_r:=\ol{M_r}$. The group $N_p\cap N_r$ is then central in $N_r$. Since abelian groups are elementary, $N_p\cap N_r=\{1\}$ because $G$ is elementary free. We have thus verified the second claim of the lemma.\par

\indent Let us now argue $F(U)_p$ is a compact open subgroup of $N_p$. 

\begin{claim*}
$F(U)_p\sleq N_p$.
\end{claim*}

\begin{proof}[Proof of claim.]
Fix $u\in F(U)_{p}$. For $C\in \mc{H}_p$ and $L\in \Uca(G)$, we may find $W\sleq_oL$ and $(W_i)_{i\in \omega}$ a normal basis at $1$ for $W$ such that 
\[
C\sleq \bigcap_{i\in \omega}\bigcap_{q\in \mb{P}\setminus\{p\}}N_G\left(F(W)_{q}\cap W_i\right).
\]

\indent Since $F(W)\sim_c F(U)$ by Lemma~\rm\ref{lem:fitt_comm}, there is $l\in \omega$ so that $F(W_l)=F(W)\cap W_l\sleq F(U)$. Setting $\tilde{W}:=W_l$, we see $\tilde{W}\sleq_oL$ and $(W_{i+l})_{i\in \omega}$ is a normal basis at $1$ for $\tilde{W}$. Additionally,
\[
F(W)_{q}\cap \tilde{W}=F(W)_{q}\cap F(\tilde{W})=F(\tilde{W})_{q},
\]
and hence,
\[
C\sleq \bigcap_{l\in \omega}\bigcap_{q\in \mb{P}\setminus\{p\}}N_G\left(F(\tilde{W})_{q}\cap W_{i+l}\right).
\]
Since $F(\tilde{W})_{q}\sleq F(U)_{q}$, the element $u$ centralizes $F(\tilde{W})_{q}$. Therefore,
\[
\cgrp{u,C}\sleq \bigcap_{l\in \omega}\bigcap_{q\in \mb{P}\setminus\{p\}}N_G\left(F(\tilde{W})_{q}\cap W_{i+l}\right) ,
\]
proving the claim.
\end{proof}

\indent We now consider $V:=U\cap N_p$. As $V\trianglelefteq U$, we infer $F(V)\sleq F(U)$, and as the group $N_p$ avoids $F(U)_q$ for any $q\neq p$, the claim implies $F(V)=F(U)_p$.  Since the group $V$ is a finite rank profinite group, $V/F(V)$ is virtually solvable. Additionally, since $G$ is elementary-free, $N_p$ has trivial elementary radical. The compact open subgroup $V$ of $N_p$ thus has trivial quasi-centre and via Theorem~\ref{thm:[A]semisimple}, has no non-trivial locally normal abelian subgroups. Lemma~\ref{lem:inf_lfr} then implies $V$ is indeed pro-$\xi$ for $\xi$ some finite set of primes. In view of Corollary~\ref{cor:melnikov}, we conclude $V$ is virtually pro-nilpotent. The group $F(V)=F(U)_p$ is therefore a compact open subgroup of $N_p$ verifying the lemma.
\end{proof}

\subsection{Simple normal Subgroups}

We here examine the topologically simple normal subgroups obtained in the previous section. Our investigations require a general lemma; the proof of which invokes a subsidiary fact from the literature.

\begin{lem}[{\cite[Lemma 5.3]{BM96}} ]\label{lem:BM}
Let $k$ be a local field, $G$ be a simply connected almost $k$-simple $k$-algebraic group, and $H$ a locally compact group. Then any non-trivial continuous homomorphism $\pi:G(k)\rightarrow H$ is proper - i.e. the preimage of any compact set is compact.
\end{lem}

With the above lemma in hand, we obtain a general fact concerning topologically simple l.c.s.c. $p$-adic Lie groups.

\begin{lem}\label{lem:p-adic_inj} 
Suppose $G$ is a t.d.l.c.s.c. group, $N$ is a non-elementary topologically simple l.c.s.c. $p$-adic Lie group, and $\psi:N\rightarrow G$ is a non-trivial continuous homomorphism. Then $\psi$ is a closed map, and $N\simeq \psi(N)$. 
\end{lem}

\begin{proof} Applying Theorem~\rm\ref{thm:simple_type}, we may assume $N=S(\Qp)^+$ for $S$ some adjoint $\Qp$-simple isotropic $\Qi{p}$-algebraic group. Take $\tilde{S}$ the simply connected covering of $S$ and let $\tilde{p}:\tilde{S}\rightarrow S$ be the covering $\Qp$-isogeny, \cite[(1.4.11)]{Mar91}. Since isogenies carry almost $\Qi{p}$-simple factors to almost $\Qi{p}$-simple factors, $\tilde{S}$ is an almost $\Qi{p}$-simple $\Qi{p}$-algebraic group.\par

\indent The group $\tilde{S}$ is simply connected, hence $\tilde{S}(\Qi{p})=\tilde{S}(\Qi{p})^+$ via \cite[(2.3.1)]{Mar91}. We thus see 
\[
\tilde{S}(\Qi{p})=\tilde{S}(\Qi{p})^+\xrightarrow{\tilde{p}}S(\Qi{p})^+\xrightarrow{\psi} G.
\]
In view of \cite[(2.3.4)]{Mar91}, $\psi\circ \tilde{p}:\tilde{S}(\Qi{p})\rightarrow G$ is a non-trivial continuous homomorphism, whereby Lemma~\rm\ref{lem:BM} implies $\psi\circ \tilde{p}$ is proper. It follows $\psi$ is also proper since $\tilde{p}\upharpoonright_{\tilde{S}(\Qp)}$ is surjective. The map $\psi:N\rightarrow G$ is therefore a closed map, and $N\simeq \psi(N)$. 
\end{proof}

As an immediate consequence of Lemma~\rm\ref{lem:p-adic_inj}, we obtain information on certain quasi-products. 

\begin{lem}\label{lem:p-adic_quasi}
Suppose $G$ is a t.d.l.c.s.c. group and $G$ is a quasi-product with quasi-factors $N$ and $M$ where $N$ is a non-elementary topologically simple $p$-adic Lie group. Then $G\simeq N\times M$.
\end{lem}

\indent We note one further consequence.

\begin{lem}\label{lem:finindex}
Suppose $G$ is a t.d.l.c.s.c. group and $N\trianglelefteq G$ is a non-elementary topologically simple $p$-adic Lie group. Then $C_G(N)N$ is a finite index closed normal subgroup of $G$.
\end{lem}

\begin{proof}
Put $H:=\ol{C_G(N)N}$. Since $N$ is a non-elementary topologically simple $p$-adic Lie group, $H$ is a quasi-product of $C_G(N)$ and $N$, so via Lemma~\rm\ref{lem:p-adic_quasi}, $C_G(N)N=H$. The subgroup $C_G(N)N$ is thus a closed normal subgroup of $G$.\par

\indent For the finite index claim, observe
\[
G/C_G(N)N\hookrightarrow \Aut(N)/\Inn(N).
\]
The group $\Aut(N)/\Inn(N)$ is finite by Theorem~\rm\ref{thm:index}, hence the group $G/C_G(N)N$ is finite verifying the lemma.
\end{proof}

\subsection{Decomposition Theorems}

\begin{thm}\label{thm:pi_lfr}
Suppose $G$ is a t.d.l.c.s.c. group that is locally of finite rank and locally pro-$\pi$ for some finite set of primes $\pi$. Then either $G$ is elementary or the ascending elementary series
\[
 \{1\}\sleq A_1\sleq A_2\sleq G
\]
is such that 
\begin{enumerate}[(1)]
\item~$A_1$ is elementary and $G/A_2$ is finite; and

\item~ $A_2/A_1\simeq N_1\times\dots\times N_k$ for $0< k<\infty$ where each $N_i$ is a non-elementary compactly generated topologically simple $p$-adic Lie group of adjoint simple type for some $p\in \pi$.
\end{enumerate}
\end{thm}

\begin{proof}
Suppose $G$ is non-elementary and put $H:=A_2/A_1$. The group $H$ is non-trivial, elementary-free, locally of finite rank, and locally pro-$\pi$. Let $N_1,\dots ,N_k$ list the minimal normal subgroups of $H$ given by Lemma~\rm\ref{lem:minlem}. Each $N_i$ is non-elementary and by Lemma~\rm\ref{lem:simplelem}, a topologically simple $p$-adic Lie group for some $p\in \pi$. Theorem~\rm\ref{thm:simple_type} implies each $N_i$ is of adjoint simple type, and appealing to \cite[(2.3.5)]{Mar91}, each $N_i$ is compactly generated. \par

\indent Lemma~\rm\ref{lem:finindex} implies $C_H(N_i)N_i=H$ for each $i$, hence $\bigcap_{i=1}^k C_H(N_i)N_i=H$. Since $N_i\sleq C_H(N_j)$ for $i\neq j$,
\[
\bigcap_{i=1}^k C_H(N_i)N_i=\left( \bigcap_{i=1}^k C_H(N_i)\right) N_1\dots N_k.
\]
The group $\bigcap_{i=1}^k C_H(N_i)$ is then a closed normal subgroup of $H$ that meets $N_i$ trivially for each $i$. The $N_i$, however, list all non-trivial minimal normal subgroups of $H$, hence $\bigcap_{i=1}^k C_H(N_i)=\{1\}$. We conclude $N_1\dots N_k=H$, and it follows $H\simeq  N_1\times\dots\times N_k$.\par

\indent The conjugation action of $G/A_1$ on the collection of minimal normal subgroups of $A_2/A_1$ induces a continuous homomorphism $\psi:G/A_1\rightarrow Sym(k)$. The group $L:=\ker(\psi)$ is then a finite index normal subgroup of $G/A_1$, and for each $N_i$, $N_i\trianglelefteq L$. Applying Lemma~\rm\ref{lem:finindex}, $C_{L}(N_i)N_i$ is finite index in $L$ for each $i$, whereby
\[
\bigcap_{i=1}^k C_L(N_i)N_i=\left(\bigcap_{i=1}^k C_L(N_i)\right) N_1\dots N_k
\]
is finite index in $L$. \par

\indent Since $G/A_1$ permutes $\{N_1,\dots, N_k\}$, the subgroup $\bigcap_{i=1}^k C_L(N_i)$ is normal in $G/A_1$. Furthermore, 
\[
\left(\bigcap_{i=1}^k C_L(N_i)\right)\simeq \left(\left(\bigcap_{i=1}^k C_L(N_i)\right) N_1\dots N_k\right)/\left(A_2/A_1\right)\sleq G/A_2,
\]
so $\bigcap_{i=1}^k C_L(N_i)$ is an elementary normal subgroup of $G/A_1$. The group $A_1$ is the elementary radical of $G$, hence $\bigcap_{i=1}^k C_L(N_i)=\{1\}$. We conclude $N_1\dots N_k=A_2/A_1$ is finite index in $G/A_1$. The group $A_2$ is therefore finite index in $G$ finishing the proof of the theorem.
\end{proof}

We obtain an immediate, interesting corollary:
\begin{cor}\label{cor:p_lfr}
Suppose $G$ is l.c.s.c. $p$-adic Lie group. Then either $G$ is elementary or the ascending elementary series
\[
\{1\}\sleq A_1\sleq A_2\sleq G
\]
is such that 
\begin{enumerate}[(1)]
\item~ $A_1$ is elementary and $G/A_2$ is finite; and
 
\item~ $A_2/A_1\simeq N_1\times\dots\times N_k$ for $0<k<\infty$ where each $N_i$ is a non-elementary compactly generated topologically simple $p$-adic Lie group of adjoint simple type.
\end{enumerate}
\end{cor}
\begin{proof} 
By Theorem~\rm\ref{thm:p-adic_char}, $G$ contains a compact open subgroup that has finite rank and is pro-$p$. The corollary then follows from Theorem~\rm\ref{thm:pi_lfr}.
\end{proof}

\begin{rmk}
For a l.c.s.c. $p$-adic Lie group $G$, $G/A_1=G/\Rad{\Es}(G)$ is semisimple as a Lie group via Proposition~\rm\ref{prop:sgrp_gl}. It is also \emph{necessarily} the case $G/\Rad{\Es}(G)$ is compactly generated. Indeed, each direct factor $N_i$ given by Corollary~\ref{cor:p_lfr} is compactly generated, and therefore, $G/\Rad{\Es}(G)$ is compactly generated.
\end{rmk}

\indent Combining Theorem~\rm\ref{thm:pi_lfr} and Corollary~\ref{cor:loc_nil_main2}, we obtain a decomposition result for locally of finite rank t.d.l.c.s.c. groups:

\begin{cor}\label{cor:finite_rank}
Suppose $G$ is a t.d.l.c.s.c. group that is locally of finite rank. Then either $G$ is elementary or $G$ admits an increasing exhaustion by compactly generated open subgroups $(P_i)_{i\in \omega}$ such that for each $i\in \omega$, the ascending elementary series for $P_i$
\[
\{1\}\sleq A_1(i)\sleq A_2(i)\sleq P_i
\]
is so that
\begin{enumerate}[(a)]
\item~$A_1(i)$ is elementary and $P_i/A_2(i)$ is finite; and

\item~$A_2(i)/A_1(i)\simeq N_1(i)\times\dots\times N_{k_i}(i)$ for $0< k_i<\infty$ where each $N_j(i)$ is a non-elementary compactly generated topologically simple $p$-adic Lie group of adjoint simple type for some prime $p$.
\end{enumerate}
\end{cor}

\indent We can indeed do better than Corollary~\ref{cor:finite_rank}. Suppose $G$ is a t.d.l.c.s.c. group that is locally of finite rank and non-elementary. By passing to $A_2/A_1$ where $\{1\}\sleq A_1\sleq A_2\sleq G$ is the ascending elementary series, we may assume $G$ is elementary-free. Fix $U$ a compact open subgroup of finite rank and let $\pi$ list the primes $p$ so that the Fitting subgroup $F(U)$ has a non-trivial $p$-Sylow subgroup. We now apply Lemma~\ref{lem:loc_finrank} to build $N_p\trianglelefteq G$ for each $p\in \pi$ so that $W_p:=F(U)_p$ is a compact open subgroup of $N_p$. \par

\indent The group $L:=\cgrp{N_p\mid p\in \pi}$ is a normal subgroup of $G$ containing $F(U)$. In view of Theorem~\ref{thm:finrank}, $G/L$ has a compact open subgroup which is a quotient of the virtually solvable group $U/F(U)$. The group $G/L$ is therefore locally solvable, and via Theorem~\ref{thm:locsolv_e}, it is elementary. Since $G$ is elementary-free, $L=G$, so $\grp{N_p\mid p\in \pi}$ is dense in $G$.\par 

\indent For each $p\in \pi$, form $K_p:=\cgrp{N_q\mid q\in \mb{P}\setminus \{p\}}$. As $N_p$ centralizes a dense subgroup of $K_p$, it follows the group $N_p\cap K_p$ is an abelian, a fortiori elementary, normal subgroup of $G$. Since $G$ is elementary-free, $N_p\cap K_p$ is trivial. The restriction of the usual projection $\phi:N_p\rightarrow G/K_p$ is thus injective, and by the previous paragraph, it has dense image.\par

\indent On the other hand, the group $N_p$ is locally pro-$p$ and locally of finite rank with trivial elementary radical. Corollary~\ref{cor:p_lfr} thus implies $N_p$ is virtually a finite direct product of non-elementary compactly generated topologically simple $p$-adic Lie groups of adjoint simple type; say $L\trianglelefteq N_p$ is the finite index direct product. Via Lemmas~\ref{lem:p-adic_inj} and \ref{lem:p-adic_quasi}, the image $\phi(L)$ is closed, and it follows further $\phi(L)$ has finite index in $G$. We conclude $L\simeq \phi(L)=G$ since $G$ is elementary-free, and therefore, $N_p=L$. That is to say, $N_p$ is a finite direct product of non-elementary compactly generated topologically simple $p$-adic Lie groups of adjoint simple type.\par

\indent We are now prepared to prove the desired theorem.
 
\begin{thm}\label{thm:loc_fin_rank}
Suppose $G$ is a t.d.l.c.s.c. group that is locally of finite rank. Then either $G$ is elementary or the ascending elementary series
\[
\{1\}\sleq A_1\sleq A_2\sleq G
\]
is such that 
\begin{enumerate}[(1)]
\item~$A_1$ is elementary and $G/A_2$ is elementary; and

\item~there is a possibly infinite set of primes $\pi$ so that $A_2/A_1$ is a quasi local direct product of $(N_p,W_p)_{p\in \pi}$ where $N_p$ is a finite direct product of non-elementary compactly generated topologically simple $p$-adic Lie groups each of adjoint simple type and $W_p$ is a compact open subgroup of $N_p$.
\end{enumerate}
\end{thm}

\begin{proof}
By passing to $A_2/A_1$, we may assume $G$ is elementary-free. Fix $U$ a compact open subgroup of finite rank and let $\pi$ list the primes $p$ so that the Fitting subgroup $F(U)$ has a non-trivial $p$-Sylow subgroup. We now apply Lemma~\ref{lem:loc_finrank} to build $N_p\trianglelefteq G$ for each $p\in \pi$ so that $W_p:=F(U)_p$ is a compact open subgroup of $N_p$. By the proceeding discussion, $N_p$ is a finite direct product of non-elementary compactly generated topologically simple $p$-adic Lie groups of adjoint simple type. It thus remains to show that $G$ is a quasi local direct product of $(N_p,W_p)_{p\in \pi}$. \par

\indent In view of Proposition~\ref{prop:pro-nil}, the multiplication map $m:\prod_{p\in \pi}W_p\rightarrow F(U)$ is an isomorphism of profinite groups. The multiplication map $m:\bigoplus_{p\in \pi}(N_p,W_p)\rightarrow G$ is therefore well-defined. Since $N_p$ centralizes $N_q$ for $p\neq q$, it follows $m$ is also a continuous homomorphism. The image of $m$ finally contains $\grp{N_p\mid p\in \pi}$, so it is dense via our preliminary discussion.\par

\indent We now check that $m$ is injective. Say $(h_{p_0},\dots,h_{p_n},w_{p_{n+1}},\dots)\mapsto 1$. This implies $M:=\cgrp{N_{p_0}\dots N_{p_n}}\cap m(\prod_{m>n}W_{p_m})$ is non-trivial. There is some prime $p_k$ with $k>n$ so that the profinite group $M$ has a non-trivial $p_k$-Sylow subgroup. The uniqueness of the $p$-Sylow subgroups of $F(U)$ imply $M\cap N_{p_k}$ is non-trivial. However, $N_{p_k}$ commutes with $M$, so $M\cap N_{p_k}$ is abelian. This is absurd since $M\cap N_{p_k}$ is then a non-trivial abelian locally normal subgroup of $G$ contradicting Theorem~\ref{thm:[A]semisimple}. The multiplication map is therefore injective.\par

\indent We have now verified that $G$ is a quasi local direct product of $(N_p,W_p)_{p\in \pi}$ proving the theorem.
\end{proof}

\indent It is easy to build examples of locally of finite rank groups $G$ so that $G/A_2$ is not finite, so Theorem~\ref{thm:loc_fin_rank} is sharp in this respect. However, we do not know of any examples in which the quasi local direct product given in Theorem~\ref{thm:loc_fin_rank} fails to be a local direct product.

\section{Application: the class $\ms{P}$}
As an application, we consider the class of all l.c.s.c. $p$-adic Lie groups for all primes $p$, denoted $\ms{P}$, and its elementary closure. 

\subsection{Preliminaries}

\begin{defn}\label{def:e_closure}
Given a class of t.d.l.c.s.c. groups $\ms{G}$, the \textbf{elementary closure} of $\ms{G}$, denoted $\ms{EG}$, is the smallest class of t.d.l.c.s.c. groups such that
\begin{enumerate}[(i)]
\item~$\ms{EG}$ contains $\ms{G}$, all second countable profinite groups, and countable discrete groups.

\item~$\ms{EG}$ is closed under group extensions of second countable profinite groups, countable discrete groups, and groups in $\ms{G}$.

\item~If $G$ is a t.d.l.c.s.c. group and $G=\bigcup_{i\in \omega}G_i$ where $(G_i)_{i\in \omega}$ is an $\subseteq$-increasing sequence of open subgroups of $G$ with $G_i\in\ms{EG}$ for each $i$, then $G\in \ms{EG}$. We say $\ms{EG}$ is closed under countable increasing unions.
\end{enumerate}
\end{defn}

Under mild additional conditions on $\ms{G}$, the elementary closure is a robust class.

\begin{thm}[{\cite[Theorem 3.20]{W_1_14}}]\label{thm:e_closure}
Suppose $\ms{G}$ is a class of t.d.l.c.s.c. groups that is closed under isomorphism of topological groups, taking closed subgroups, and taking Hausdorff quotients. Suppose further $\ms{G}$ satisfies the following:
\begin{enumerate}[(a)] 
\item~If $H$ is a t.d.l.c.s.c. group and $ \psi:H\rightarrow G$ is a continuous, injective homomorphism with $G\in \ms{G}$, then $H\in \ms{EG}$.

\item~If $G$ is a t.d.l.c.s.c. group, $H,L\trianglelefteq G$, $[H,L]=\{1\}$, and $H\in \ms{G}$, then $\ol{HL}/L\in \ms{EG}$.

\end{enumerate}
Then the permanence properties in Theorem~\rm\ref{thm:closure_main} hold of $\ms{EG}$.
\end{thm}
We call classes that satisfy the hypotheses of Theorem~\rm\ref{thm:e_closure} \textbf{elementarily robust}.

\subsection{The elementary closure of $\ms{P}$}

\begin{thm}
The class $\ms{P}$ is elementarily robust. In particular, $\ms{EP}$ enjoys the same closure properties as in Theorem~\rm\ref{thm:closure_main}.
\end{thm}
\begin{proof} It is well-known the class $\ms{P}$ is closed under isomorphism of topological groups, taking closed subgroups, and Hausdorff quotients; one may, alternatively, devise a proof via Theorem~\ref{thm:p-adic_char}. We thus need to show $\ms{P}$ satisfies $(a)$ and $(b)$ of Theorem~\rm\ref{thm:e_closure}. \par

\indent For $(a)$, suppose $H$ is a t.d.l.c.s.c. group and $\psi:H\rightarrow G$ is a continuous, injective homomorphism with $G\in \ms{P}$. Let $U\in \Uca(G)$ be a pro-$p$ compact open subgroup with finite rank given by Theorem~\rm\ref{thm:p-adic_char}. Since $\psi$ is continuous, there is $W\in \Uca(H)$ such that $W\simeq \psi(W)\sleq U$. We conclude $W$ has finite rank and is pro-$p$, and via Theorem~\rm\ref{thm:p-adic_char}, $H\in \ms{EP}$.\par

\indent For $(b)$, suppose $G$ is a t.d.l.c.s.c. group, $H,L\trianglelefteq G$, $[H,L]=\{1\}$, and $H\in \ms{P}$. We may assume $H$ is non-elementary since else we are done via \cite[Lemma 3.11]{W_1_14}.\par

\indent Let $\{1\}\sleq A_1\sleq A_2\sleq H$ be the ascending elementary series for $H$ and let
\[
\psi:H/A_1\rightarrow \left(\ol{HL}/L\right)/\left(\ol{A_1L}/L\right)\simeq\ol{HL}/\ol{A_1L}
\]
be the obvious map. Since $A_1$ is elementary, $\ol{A_1L}/L$ is elementary via \cite[Lemma 3.11]{W_1_14}. Theorem~\rm\ref{thm:closure_main} then implies $\ker(\psi)$ is elementary, and since $A_1$ is the elementary radical of $H$, we infer $\ker(\psi)=\{1\}$. The map $\psi$ is therefore injective. \par

\indent Corollary~\rm\ref{cor:p_lfr} implies $A_2/A_1\simeq N_1\times \dots \times N_k$ with each $N_i$ a non-elementary topologically simple $p$-adic Lie group. In view of Lemma~\rm\ref{lem:p-adic_inj}, the group $\psi(N_i)$ is a closed subgroup of $\ol{HL}/\ol{A_1L}$ isomorphic to $N_i$. We infer $\ol{\psi(N_1)\dots \psi(N_k)}$ is a subgroup of $\ol{HL}/\ol{A_1L}$ that is a quasi-product with quasi-factors $\psi(N_1)\dots \psi(N_k)$. It now follows from Lemma~\rm\ref{lem:p-adic_quasi} $\psi(N_1)\dots \psi(N_k)$ is closed in $\ol{HL}/\ol{A_1L}$. \par

\indent The group $\psi(A_2/A_1)$ is thus a closed subgroup of $\ol{HL}/\ol{A_1L}$. Furthermore, since $A_2/A_1$ is finite index in $H/A_1$, $\psi(H/A_1)$ is closed, whereby $\psi(H/A_1)=\ol{HL}/\ol{A_1L}$. We conclude $\ol{HL}/L$ is elementary-by-$p$-adic, and therefore, $\ol{HL}/L\in \ms{EP}$ as required.
\end{proof}

\begin{rmk}
One can define a rank on $\ms{EP}$ analogous to the construction rank on $\ms{E}$, the class of elementary groups; this rank is a useful tool to study elements of $\ms{EP}$. It follows by induction on this rank that every non-discrete compactly generated group in $\ms{EP}$ that is topologically simple is a $p$-adic Lie group for some prime $p$. On the other hand, there are many non-discrete compactly generated groups that are topologically simple that are \textit{not} $p$-adic Lie groups. In view of the closure properties of $\ms{EP}$, we conclude the class of t.d.l.c.s.c. groups cannot be approximated in any reasonable way by $p$-adic Lie groups.
\end{rmk}

\bibliographystyle{bibgen}
\bibliography{C:/Users/Phillip/Dropbox/Research/Tex/Bibtex/biblio}

\def\cprime{$'$} \def\cprime{$'$} \def\cprime{$'$}
\begin{thebibliography}{10}

\bibitem{BEW11}
{\sc Y.~Barnea}, {\sc M.~Ershov}, and {\sc T.~Weigel}, Abstract commensurators
  of profinite groups. {\em Trans. Amer. Math. Soc.\/} 363, no.~10, (2011),
  5381--5417.

\bibitem{BT73}
{\sc A.~Borel} and {\sc J.~Tits}, Homomorphismes ``abstraits'' de groupes
  alg\'ebriques simples. {\em Ann. of Math. (2)\/} 97, (1973), 499--571.

\bibitem{Bo98}
{\sc N.~Bourbaki}, {\em Lie groups and {L}ie algebras. {C}hapters 1--3\/},
  Elements of Mathematics (Berlin),  (Springer-Verlag, Berlin1998). Translated
  from the French, Reprint of the 1989 English translation.

\bibitem{BM96}
{\sc M.~Burger} and {\sc S.~Mozes}, {${\rm CAT}$}(-{$1$})-spaces, divergence
  groups and their commensurators. {\em J. Amer. Math. Soc.\/} 9, no.~1,
  (1996), 57--93.

\bibitem{BM00}
{\sc M.~Burger} and {\sc S.~Mozes}, Groups acting on trees: from local to
  global structure. {\em Inst. Hautes \'Etudes Sci. Publ. Math.\/} , no.~92,
  (2000), 113--150 (2001).

\bibitem{CM11}
{\sc P.-E. Caprace} and {\sc N.~Monod}, Decomposing locally compact groups into
  simple pieces. {\em Math. Proc. Cambridge Philos. Soc.\/} 150, no.~1, (2011),
  97--128.

\bibitem{CRW13}
{\sc P.-E. Caprace}, {\sc C.~Reid}, and {\sc G.~Willis}, Locally normal
  subgroups of simple locally compact groups. ArXiv:1303.6755 [math.GR],
  http://arxiv.org/abs/1303.6755.

\bibitem{CRW_1_13}
{\sc P.-E. Caprace}, {\sc C.~Reid}, and {\sc G.~Willis}, Locally normal
  subgroups of totally disconnected groups. {P}art {I}: {G}eneral theory.
  ArXiv:1304.5144 [math.GR], http://arxiv.org/abs/1304.5144.

\bibitem{CRW_2_13}
{\sc P.-E. Caprace}, {\sc C.~Reid}, and {\sc G.~Willis}, Locally normal
  subgroups of totally disconnected groups. {P}art {II}: {C}ompactly generated
  simple groups. ArXiv:1401.3142 [math.GR], http://arxiv.org/abs/1401.3142.

\bibitem{CCLTV11}
{\sc R.~Cluckers}, {\sc Y.~Cornulier}, {\sc N.~Louvet}, {\sc R.~Tessera}, and
  {\sc A.~Valette}, The {H}owe-{M}oore property for real and {$p$}-adic groups.
  {\em Math. Scand.\/} 109, no.~2, (2011), 201--224.

\bibitem{GW01}
{\sc H.~Gl{\"o}ckner} and {\sc G.~Willis}, Uniscalar {$p$}-adic {L}ie groups.
  {\em Forum Math.\/} 13, no.~3, (2001), 413--421.

\bibitem{HR79}
{\sc E.~Hewitt} and {\sc K.~Ross}, {\em Abstract harmonic analysis. {V}ol.
  {I}\/}, vol. 115 of {\em Grundlehren der Mathematischen Wissenschaften
  [Fundamental Principles of Mathematical Sciences]\/},  (Springer-Verlag,
  Berlin1979), second edn.

\bibitem{L65}
{\sc M.~Lazard}, Groupes analytiques {$p$}-adiques. {\em Inst. Hautes \'Etudes
  Sci. Publ. Math.\/} , no.~26, (1965), 389--603.

\bibitem{LM87}
{\sc A.~Lubotzky} and {\sc A.~Mann}, Powerful {$p$}-groups. {II}. {$p$}-adic
  analytic groups. {\em J. Algebra\/} 105, no.~2, (1987), 506--515.

\bibitem{Mar91}
{\sc G.~A. Margulis}, {\em Discrete subgroups of semisimple {L}ie groups\/},
  vol.~17 of {\em Ergebnisse der Mathematik und ihrer Grenzgebiete (3) [Results
  in Mathematics and Related Areas (3)]\/},  (Springer-Verlag, Berlin1991).

\bibitem{M96}
{\sc O.~V. Mel{\cprime}nikov}, Profinite groups with finitely generated {S}ylow
  subgroups. {\em Dokl. Akad. Nauk Belarusi\/} 40, no.~6, (1996), 34--37, 123.

\bibitem{Plat66}
{\sc V.~P. Platonov}, Locally projectively nilpotent subgroups and nilelements
  in topological groups. {\em Izv. Akad. Nauk SSSR Ser. Mat.\/} 30, (1966),
  1257--1274.

\bibitem{R13}
{\sc C.~Reid}, The generalised pro-fitting subgroup of a profinite group. {\em
  Comm. Algebra\/} 41, no.~1, (2013), 294--308.

\bibitem{RZ00}
{\sc L.~Ribes} and {\sc P.~Zalesskii}, {\em Profinite groups\/}, vol.~40 of
  {\em Ergebnisse der Mathematik und ihrer Grenzgebiete. 3. Folge. A Series of
  Modern Surveys in Mathematics [Results in Mathematics and Related Areas. 3rd
  Series. A Series of Modern Surveys in Mathematics]\/},  (Springer-Verlag,
  Berlin2010), second edn.

\bibitem{Se64}
{\sc J.-P. Serre}, {\em Lie algebras and {L}ie groups\/}, vol. 1500 of {\em
  Lecture Notes in Mathematics\/},  (Springer-Verlag, Berlin1992), second edn.
  1964 lectures given at Harvard University.

\bibitem{Ti74}
{\sc J.~Tits}, {\em Buildings of spherical type and finite {BN}-pairs\/},
  Lecture Notes in Mathematics, Vol. 386,  (Springer-Verlag, Berlin1974).

\bibitem{W_1_14}
{\sc P.~Wesolek}, Elementary totally disconnected locally compact groups.
  ArXiv:1405.4851 [math.GR], http://arxiv.org/abs/1405.4851.

\bibitem{Will07}
{\sc G.~Willis}, Compact open subgroups in simple totally disconnected groups.
  {\em J. Algebra\/} 312, no.~1, (2007), 405--417.

\bibitem{Wil98}
{\sc J.~Wilson}, {\em Profinite groups\/}, vol.~19 of {\em London Mathematical
  Society Monographs. New Series\/},  (The Clarendon Press Oxford University
  Press, New York1998).

\end{thebibliography}
\end{document}